\documentclass [11pt,a4paper]{article}
\usepackage[ansinew]{inputenc}
\usepackage{amssymb}
\usepackage{amsmath}
\usepackage{graphicx}
\usepackage{amsthm}
\usepackage{cite}
\usepackage{xcolor}
\newtheorem{theorem}{\textbf{Theorem}}[section]
\newtheorem{lemma}{\textbf{Lemma}}[section]
\newtheorem{proposition}{\textbf{Proposition}}[section]
\newtheorem{corollary}{\textbf{Corollary}}[section]
\newtheorem{remark}{\textbf{Remark}}[section]
\newtheorem{definition}{\textbf{Definition}}[section]

\allowdisplaybreaks[4]

\def\be{\begin{equation}}
\def\ee{\end{equation}}
\def\bea{\begin{eqnarray}}
\def\eea{\end{eqnarray}}
\def\bt{\begin{theorem}}
\def\et{\end{theorem}}
\def\bl{\begin{lemma}}
\def\el{\end{lemma}}
\def\br{\begin{remark}}
\def\er{\end{remark}}
\def\bp{\begin{proposition}}
\def\ep{\end{proposition}}
\def\bc{\begin{corollary}}
\def\ec{\end{corollary}}
\def\bd{\begin{definition}}
\def\ed{\end{definition}}
\def\uu{\mathbf{u}}
 \def\non{\nonumber }

\begin{document}

\title{Well-posedness of a Hydrodynamic Phase-field Model for Functionalized Membrane--Fluid Interaction}
\author{
{\sc Hao Wu} \footnote{School of Mathematical Sciences and Shanghai
Key Laboratory for Contemporary Applied Mathematics, Fudan
University, 200433 Shanghai, China, Email:
\textit{haowufd@fudan.edu.cn}.}\ \  and {\sc Yuchen Yang}
\footnote{School of Mathematical Sciences, Fudan
University, 200433 Shanghai, China, Email:
\textit{18110180027@fudan.edu.cn}.}}
\date{\today}
\maketitle


\begin{abstract}
In this paper, we study a hydrodynamic phase-field system modeling the deformation of functionalized membranes in incompressible viscous fluids. The governing PDE system consists of the Navier--Stokes equations coupled with a convective sixth-order Cahn--Hilliard type equation driven by the functionalized Cahn--Hilliard free energy, which describes phase separation in mixtures with an amphiphilic structure. In the three dimensional case, we first prove existence of global weak solutions provided that the initial total energy is finite. Then we establish uniqueness of weak solutions under suitable regularity assumptions only imposed on the velocity field (or its gradient). Finally, we prove the existence and uniqueness of local strong solutions for arbitrary regular initial data and derive some blow-up criteria. The results are obtained in the general setting with variable viscosity and mobility.



\end{abstract}

\section{Introduction}\setcounter{equation}{0}
The study of formation and dynamics of biological membranes (e.g., vesicle) is an important and interesting subject in biology and bioengineering. Roughly speaking, the biomembranes consist of a bilayer of amphiphilic molecules (lipids). They separate the interior of cells from the outside environment and divides intracellular regions into different functional compartments. In various mechanical, physical and biological environments, the shape of membranes can exhibit rich geometric structures, which plays an important role in their biological function \cite{Sei}. A lot of work has been done to understand morphological changes of membranes. We refer to \cite{Al17,Al14,Ca06,Ca70,DLRW05,DLRW05a,DLRW09, DLW04, DLW06,HE,Lo09} for mathematical modeling and  numerical simulations, see also \cite{CW20,CG19,CL12,DLL07,EB15,LYN,WX13} for rigorous analysis.

The equilibrium configurations of membranes can be characterized by the Helfrich bending elasticity energy \cite{Ca70,HE}.
In the simplest isotropic case, if the evolution of the membrane does not change its topology, the bending energy takes the following
simplified form:
 \be
E_{\mbox{elastic}}=\int_{\Gamma}\frac{k}{2}(H-H_0)^2\,\mathrm{d}S, \non
  \ee
where $k$ is the bending rigidity, $H$ is the mean curvature of the membrane surface and $H_0$ is the spontaneous
curvature that describes certain physical/chemical difference between the inside and the outside of the membrane. For a homogeneous bilayer lipid membrane $H_0=0$.

In order to understand the complex non-equilibrium motion and shape deformation of biomembranes, the coupling of the dynamics with the surrounding fluid has to be taken into account. In this aspect, the phase field (or diffuse interface) method provides a simple and efficient way to incorporate physical effects at different scales \cite{DLW04,Lo09,YFLS}. Let $\phi$ be the phase function defined on the physical domain $\Omega$, which indicates the inside and outside of the membrane $\Gamma$ by taking distinct values $1$ and $-1$, respectively. Then the Helfrich energy $E_{\mbox{elastic}}$ is approximated by a modified Willmore energy (see \cite{DLW04,DLRW05})
 \be\non
 E_\varepsilon(\phi)=\dfrac{k}{2\varepsilon}\int_{\Omega}
 \left(-\varepsilon \Delta \phi+ \frac{1}{\varepsilon}(\phi^2-1)\phi\right)^2\,\mathrm{d}x,
 \ee
where $\varepsilon$ is a small positive parameter (compared to the vesicle size) that characterizes the transition layer of the phase function.
We refer to \cite{DLRW05,Wang08} for the convergence of the phase-field model to the original sharp interface model as the transition width of the diffuse interface $\varepsilon \rightarrow 0$. Since the biomembrane is usually a closed surface (so that change in volume is much slower than the change of shape) and the presence of the lipid molecules yields certain inextensibility of the membrane, some constraints on its volume and surface area have to be imposed. In the phase-field framework, the volume and surface area of are approximated by \cite{DLW04}:
$$
A(\phi)=\int_\Omega \phi\,\mathrm{d}x,\quad B(\phi)=\int_\Omega \Big[\frac{\varepsilon}{2} |\nabla \phi|^2 + \frac{1}{4\varepsilon}(\phi^2-1)^2\Big] \,\mathrm{d}x.
$$
Lagrange multipliers should be introduced to preserve the volume and surface area \cite{DLRW09,DLW06}.
In an alternative manner, suitable penalty terms are introduced in the energy to enforce the constraints \cite{DLRW05, DLRW05a,DLW06}:
 \begin{equation}\non
 E_{\text{penalty}}(\phi)=E_\varepsilon(\phi)+\frac12 M_1(A(\phi)-\alpha)^2+\frac12M_2(B(\phi)-\beta)^2,\quad M_1,M_2>0,
 \end{equation}
where $\alpha=A(\phi_0)$ and $\beta=B(\phi_0)$ are determined by the initial value of the
phase function $\phi_0$. For numerical and analytical studies of the phase-field model for membrane without the coupling with the hydrodynamics of the surrounding fluids, we refer to \cite{Ca06,CL11,CL12,DLRW05,DLRW05a,DLW04,DLW06,Wang08} and the references cited therein.

When the fluid interaction is considered, the situation turns out to be more complicated. Several fluid-phase-field coupling systems have been derived, e.g., via the variational principle, to investigate the complicated interactions between membrane elasticity and hydrodynamic forces  \cite{Al14,DLRW09,GKL18,Lo16}. Here, we only focus on some related analytical studies. In \cite{DLL07}, the authors studied a hydrodynamical system in which the phase-field function is governed by an advected Allen--Cahn type equation for the penalized energy $E_{\text{penalty}}(\phi)$. For the initial boundary value problem subject to a no-slip boundary condition for the velocity field and Dirichlet boundary conditions for the phase function, they proved obtained the existence of global weak solutions by using the Galerkin method and showed the uniqueness of solutions in a more regular class, see \cite{EB15} for extensions to a regularized version of the problem involving the $\alpha$-Navier--Stokes equations. Well-posedness of strong solutions was investigated in \cite{Ka18,LYN}. In the periodic setting, further results on the regularity, uniqueness and long-time behavior of solutions have been obtained in \cite{WX13}, see also \cite{Zhao13,Zhao14,Zhao15} for some results on the uniqueness and blow-up criteria. However, we note that when both the volume and surface area constraints are imposed by Lagrange multipliers, analysis of the resulting hydrodynamic system (see \cite{DLRW09}) remains open. On the other hand, the authors of \cite{CG18,CG19} analyzed an alternative model, in which the phase-field function is governed by a convective Cahn--Hilliard type equation so that the volume conservation can be automatically guaranteed. Adopting the penalty of surface area constraint, the authors proved eventual regularity of global weak solutions and showed the convergence to equilibrium as times goes to infinity by using the \L ojasiewicz--Simon approach.

In this paper, we consider a hydrodynamic phase-field model that describes the deformation of functionalized membranes in an incompressible viscous fluid, or the phase separation process of incompressible viscous two-phase  flows with an amphiphilic structure. The resulting system of partial differential equations reads as follows
\begin{align}
  &\partial_t \uu +(\uu\cdot\nabla) \uu-\nabla \cdot (2\nu(\phi)D \uu) +\nabla P =\mu \nabla\phi,
  \label{NS}\\
  &\nabla\cdot \uu=0, \label{incom}\\
  &\partial_t \phi+\uu\cdot\nabla\phi=\nabla \cdot(m(\phi)\nabla \mu),  \label{phasefield}\\
  &\mu=- \Delta \omega + f'(\phi) \omega  + \eta \omega, \label{pot1}\\
  &\omega= - \Delta \phi + f(\phi), \label{pot2}
\end{align}
in $\Omega\times(0,T)$, where $\Omega \subset\mathbb{R}^3$ is a bounded domain with smooth boundary $\partial\Omega$ and $T>0$.
Here, the order parameter $\phi$ denotes the difference in volume fractions of the binary mixture. The fluid velocity $\uu$ is taken as the volume-averaged velocity with $D\uu=\frac{1}{2}(\nabla\uu+\nabla^ \mathrm{T}\uu)$ being the symmetrized velocity gradient, and the scalar function $P$ stands for the (modified) pressure. The system \eqref{NS}--\eqref{pot2} is subject to the following boundary conditions
\begin{alignat}{3}
&\uu=\mathbf{0},\quad {\partial}_{\mathbf{n}}\phi={\partial}_{\mathbf{n}}\Delta\phi= {\partial}_{\mathbf{n}}\mu=0,\qquad\qquad &\textrm{on}& \   \partial\Omega\times(0,T),
\label{boundary}
\end{alignat}
and the initial conditions
\begin{alignat}{3}
&\uu|_{t=0}=\uu_{0}(x),\ \ \phi|_{t=0}=\phi_{0}(x),\qquad &\textrm{in}&\ \Omega.
\label{IC}
\end{alignat}
In \eqref{boundary}, $\mathbf{n}=\mathbf{n}(x)$ denotes the unit outward normal vector on $\partial\Omega$.

Different from the models for vesicle membranes that have been analyzed in the literature, our system is driven by the following functionalized Cahn--Hilliard free energy
\begin{equation}\label{defiE}
  E(\phi) = \int_\Omega \Big[ \frac12 \big(- \Delta \phi + f(\phi) \big)^2 + \eta \Big( \frac12 |\nabla \phi|^2 + F(\phi) \Big) \Big]\, \mathrm{d}x.
\end{equation}
where $f(\phi)=\phi^3-\phi$ and $F(\phi)=\frac{1}{4}(\phi^2-1)^2$. For the sake of simplicity, we have set $\varepsilon=1$ in the expression of \eqref{defiE} since its value does not play a role in the subsequent mathematical analysis. The energy $E(\phi)$ was derived in \cite{GS90} to model phase separation of mixtures with an amphiphilic structure. Later on, it was extended to describe nanoscale morphology changes in functionalized polymer chains \cite{PW09} and bilayer membranes \cite{GHPL,DP13,DP15}. Instead of minimizing or preserving the surface area, functionalized materials (membranes) have embedded charged groups which interact exothermically with polar solvents, spontaneously generating surfactant loaded interfaces \cite{GHPL}. The parameter $\eta<0$ in \eqref{defiE} reflects the competition between the square of variational derivative of the classical Cahn--Hilliard free energy $E_{\text{CH}}(\phi)= \int_\Omega \frac12 |\nabla \phi|^2 + F(\phi)\,\mathrm{d}x$ against itself, which implies balancing of the elastic deformation against hydrophilic surface interactions. The functionalized Cahn--Hilliard energy can incorporate the propensity of the amphiphilic surfactant phase to drive the creation of interfaces and naturally produce stable bilayers, or homoclinic interfaces with an intrinsic width \cite{DP13}. Minimization problems, bilayer structures, pearled patterns, and network bifurcations related to the functionalized Cahn--Hilliard energy have been extensively studied, see \cite{DP15,PZ13,PW17} and the references therein. When the fluid coupling is neglected (i.e., setting $\uu=\mathbf{0}$ in \eqref{phasefield}), for the functionalized Cahn--Hilliard equation subject to periodic boundary conditions, the authors of \cite{DLP19} proved the existence of global weak
solutions in the case of a regular potential and a degenerate mobility. For $\eta\in \mathbb{R}$, $m=1$, and a regular potential, existence and uniqueness of global solutions in the Gevrey class were established in \cite{CW20}, again in the periodic setting. In the recent work \cite{SW20}, the authors considered the case $\eta\in\mathbb{R}$, $m=1$ with a physically relevant logarithmic potential $F(\phi)= \frac12 (1 + \phi) \ln  (1 + \phi) + \frac12 (1 - \phi) \ln (1 - \phi) - \frac{\lambda}{2} \phi^2$, $\phi\in (-1,1)$. By overcoming the difficulty from singular diffusion terms, the authors proved existence, uniqueness and regularity of global weak solutions. Besides, they proved existence of the global attractor for the associated dynamical process in a suitable complete metric space.

To the best of our knowledge, there seems no analytical result for the hydrodynamic problem \eqref{NS}--\eqref{IC} in the literature.
Thus, the aim of this paper is to perform a first step analysis on its well-posedness, under a general and physically reasonable setting with variable fluid viscosity and mobility. Based on the energy dissipation structure of problem \eqref{NS}--\eqref{IC} (see \eqref{basic energy law} below), we are able to prove the existence of global weak solutions (see Theorem \ref{thm:weak}) by employing a suitable Galerkin approximation scheme. After that, we derive some sufficient conditions under which a weak solution can be unique (see Theorem \ref{thm:weakuni}). Those uniqueness criteria coincide with the results for the conventional Navier--Stokes equations \cite{Be02,Ri02,Se62,SH84}. Hence, we infer that in spite of the nonlinear coupling between the equations for the velocity field $\uu$ and the phase function $\phi$, the velocity field turns out to play a dominant role in the study of uniqueness property for weak solutions to problem \eqref{NS}--\eqref{IC}. Next, we show existence and uniqueness of local strong solutions to the problem \eqref{NS}--\eqref{IC} (see Theorem \ref{thm:locstr}). The proof relies on some specific higher-order differential inequalities inspired by previous works on the Navier--Stokes--Cahn--Hilliard system for incompressible two-phase flows, see e.g., \cite{GMT18}. Finally, we establish some blow-up criteria for local strong solutions to problem \eqref{NS}--\eqref{IC}, which again only involve the velocity field or its gradient (see Theorems \ref{thm:buc}).

There are several issues remain open for problem \eqref{NS}--\eqref{IC}. For instance, the existence of a unique global strong solution under certain additional assumption either on the lower bound for the fluid viscosity (i.e., large viscosity) or on the initial data (i.e., near equilibrium). Moreover, it will be interesting to investigate the long-time behavior of global solutions. In particular, the convergence to a single equilibrium $(\mathbf{0},\phi_\infty)$ as $t\to +\infty$ for any bounded global weak/strong solutions, as well as Lyapunov stability of the zero velocity field and local (or global) minimizers of the elastic bending energy $E(\phi)$ (cf. \cite{CG18,CG19,WX13} for related studies on some different type of vesicle-fluid interaction models). In this aspect, a suitable gradient inequality of \L ojasiewicz--Simon type associated to the energy functional $E(\phi)$ will play a crucial role.

The remaining part of the paper is organized as follows. In Section 2, we introduce the functional settings and state the main results.
In Section 3, we prove the existence of global weak solutions and establish some uniqueness criteria only in terms of the velocity field.
In Section 4, we prove the existence as well as uniqueness of local strong solutions and then derive some blow-up criteria.
In the appendix, we present a formal physical derivation of the hydrodynamic system \eqref{NS}--\eqref{IC} via the energetic variational approach \cite{BKL16,GKL18,HKL,LW19}.

\section{Main Results}\setcounter{equation}{0}

 \subsection{Preliminaries}
 Let $X$ be a real Banach or Hilbert space. Its dual space is denoted by $X^*$, and the duality pairing between $X$ and $X^*$ is denoted by $\langle \cdot,\cdot\rangle_{X^*,X}$. Given an interval $I$ of $\mathbb{R}^+$, we introduce the function space $L^p(I;X)$ with $p\in [1,+\infty]$, which consists of Bochner measurable $p$-integrable functions with values in $X$. The boldface letter $\mathbf{X}$ denotes the space for vector (or matrix) valued functions. We shall make use the Frobenius inner product $M_1 : M_2 = \mathrm{trace}(M_1^\mathrm{T}M_2)$ for two arbitrary $3\times 3$ matrices $M_1, M_2$ and the Frobenius norm $|M|^2=M:M$ for $M\in \mathbb{R}^{3\times 3}$.

  Throughout this paper, we assume that $\Omega \subset\mathbb{R}^3$ is a bounded domain with sufficiently smooth boundary $\partial\Omega$. For the standard Lebesgue and Sobolev spaces, we use the notations $L^{p} := L^{p}(\Omega)$ and $W^{k,p} := W^{k,p}(\Omega)$ for any $p \in [1,+\infty]$, $k > 0$ equipped with the norms $\|\cdot\|_{L^{p}}$ and $\|\cdot\|_{W^{k,p}}$.  In the case $p = 2$ we denote $H^{k} := W^{k,2}$ and its norm by $\|\cdot\|_{H^{k}}$.
  For non integer values of $s$, $s=\theta k+(1-
  \theta)(k+1)$ for some $\theta\in (0,1)$, $W^{s,p}(\Omega)$ denotes the (complex) interpolation space $[W^{k,p}(\Omega),W^{k+1,p}(\Omega)]_\theta$.
  The norm and inner product on $L^{2}(\Omega)$ are simply denoted by $\|\cdot\|$ and $(\cdot,\cdot)$, respectively. We recall here some useful interpolation inequalities in three dimensions that will be frequently used in the subsequent analysis (see e.g., \cite{Te,Zh}).
\begin{lemma}
Suppose that $\Omega\in \mathbb{R}^3$ is a bounded domain with smooth boundary.
\begin{itemize}
\item[$(1)$] Ladyzhenskaya's inequality. $\|f\|_{L^4}\leq C\|f\|^\frac14\|f\|_{H^1}^\frac34$, for any $f\in H^1(\Omega)$.
\item[$(2)$] Agmon's inequality. $\|f\|_{L^\infty}\leq C\|f\|_{H^1}^\frac12\|f\|_{H^2}^\frac12$, for any $f\in H^2(\Omega)$.
\item[$(3)$] Gagliardo--Nirenberg inequality. Let $j, m$ be arbitrary integers satisfying $0\leq j< m$ and let $1\leq q, r\leq +\infty$, $\frac{j}{m}\leq a\leq 1$ such that
\[
\frac{1}{p}-\frac{j}{3}=a\left(\frac{1}{r}-\frac{m}{3}\right)+(1-a)\frac{1}{q}.
\]
 For any $f\in W^{m,r}(\Omega)\cap L^q(\Omega)$, there are two constants $c_1, c_2$ independent of $f$ such that
\[
\|\partial^jf\|_{L^p}\leq c_1\|\partial^m f\|_{L^r}^a\|f\|_{L^q}^{1-a}+c_2\|f\|_{L^q},
\]
with the following exception: if $1<r<+\infty$ and $m-j-\frac{3}{r}$ is a nonnegative integer, then the above inequality holds only for $\frac{j}{m}\leq a<1$.
\end{itemize}
\end{lemma}
\noindent We also report the following estimate on the product of functions in three dimensions:
$$
\|fg\|_{H^m}\leq C(\|f\|_{L^\infty}\|g\|_{H^m}+\|f\|_{H^m}\|g\|_{L^\infty}),\quad \forall\,f,g\in H^m(\Omega), \ m\geq 2.
$$

For every $f\in (H^1(\Omega))^*$, we denote by $\overline{f}$ its  generalized mean value over $\Omega$ such that
$\overline{f}=|\Omega|^{-1}\langle f,1\rangle_{(H^1)^*,H^1}$; if $f\in L^1(\Omega)$, then $\overline{f}=|\Omega|^{-1}\int_\Omega f \,\mathrm{d}x$.
We introduce the space $L^2_{0}(\Omega):=\{f\in L^2(\Omega):\overline{f} =0\}$ for $L^2$ functions with zero mean. In view of the homogeneous Neumann boundary condition \eqref{boundary}, we also set
\begin{align*}
&H^2_{N}(\Omega):=\{f\in H^2(\Omega):\,\partial_{\mathbf{n}}f=0 \ \textrm{on}\  \partial \Omega\},\\
&H^4_{N}(\Omega):=\{f\in H^4(\Omega):\,\partial_{\mathbf{n}}f=\partial_{\mathbf{n}}\Delta f=0 \ \textrm{on}\  \partial \Omega\}.
\end{align*}
Consider the realization of the minus Laplacian with homogeneous Neumann boundary condition
$\mathcal{A}_N\in \mathcal{L}(H^1(\Omega),(H^1(\Omega))^*)$ defined by
\begin{equation}\nonumber
   \langle \mathcal{A}_N u,v\rangle_{(H^1)^*,H^1} := \int_\Omega \nabla u\cdot \nabla v \, \mathrm{d}x,\quad \text{for }\,u,v\in H^1(\Omega).
\end{equation}
It easily follows that for the linear spaces
$$
V_0:=\{ u \in H^1(\Omega):\ \overline{u}=0\}, \quad
V_0^*:= \{ u \in (H^1(\Omega))^*:\ \overline{u}=0 \},
$$
the restriction of $\mathcal{A}_N$ from $V_0$ onto $V_0^*$
is an isomorphism. Moreover, $\mathcal{A}_N$ is positively defined on $V_0$ and self-adjoint. We denote its inverse map by $\mathcal{N} =\mathcal{A}_N^{-1}: V_0^*
\to V_0$. Note that for every $f\in V_0^*$, $u= \mathcal{N} f \in V_0$ is the unique weak solution of the Neumann problem
$$
\begin{cases}
-\Delta u=f, \quad \text{in} \ \Omega,\\
\partial_{\mathbf{n}} u=0, \quad \ \  \text{on}\ \partial \Omega.
\end{cases}
$$
Besides, it holds
\begin{align*}
&\langle \mathcal{A}_N u, \mathcal{N} g\rangle_{V^*_0,V_0} =\langle  g,u\rangle_{(H^1)^*,H^1}, && \forall\, u\in V, \ \forall\, g\in V_0^*,\\
&\langle  g, \mathcal{N} f\rangle_{V_0^*,V_0}
=\langle f, \mathcal{N} g\rangle_{V_0^*,V_0} = \int_{\Omega} \nabla(\mathcal{N} g)
\cdot \nabla (\mathcal{N} f) \, \mathrm{d}x, && \forall \, g,f \in V_0^*,
\end{align*}
and the chain rule
\begin{align}
&\langle \partial_t u, \mathcal{N} u\rangle_{V_0^*,V_0}=\frac{1}{2}\frac{\mathrm{d}}{\mathrm{d}t}\|\nabla \mathcal{N} u\|^2,\ \ \textrm{a.e. in}\ (0,T),\nonumber
\end{align}
for any $u\in H^1(0,T; V_0^*)$. For any $f\in V_0^*$, we set $\|f\|_{V_0^*}=\|\nabla \mathcal{N} f\|$.
It is easy to check that $f \to \|f\|_{V_0^*}$ and $
f \to(\|f-\overline{f}\|_{V_0^*}^2+|\overline{f}|^2)^\frac12$ are equivalent norms on $V_0'$ and $(H^1(\Omega))^*$, respectively.
We recall the well known Poincar\'{e}--Wirtinger inequality:
\begin{equation*}
\|f-\overline{f}\|\leq C_P\|\nabla f\|,\quad \forall\,
f\in H^1(\Omega),
\end{equation*}
where $C_P$ is a constant depending only on $\Omega$. Then it easily follows that $f\to \|\nabla f\|$,  $f\to (\|\nabla f\|^2+|\overline{f}|^2)^\frac12$ are equivalent norms on $V_0$ and $H^1(\Omega)$, respectively. Moreover, we have the following estimates
\begin{align*}
&\|f\|\leq C\|f\|_{V_0^*}^\frac12\|\nabla f\|^\frac12, &&\forall\, f\in V_0,\\
&\|\mathcal{N}f\|_{H^{m+2}}\leq C\|f\|_{H^m}, &&\forall\, f\in H^m(\Omega)\cap L^2_0(\Omega),\ m\geq 0,
\end{align*}
where $C>0$ only depends on $\Omega$.

Concerning the Navier--Stokes equations, we introduce the following spaces (see e.g., \cite{G,Te})
\begin{align*}
	& \mathbf{H}_\sigma := \{ \mathbf{u} \in \mathbf{L}^2(\Omega): \nabla \cdot \mathbf{u} = 0 \text{ in } \Omega, \: \mathbf{u} \cdot \mathbf{n} = 0 \text{ on } \partial\Omega\}, \\
&  \mathbf{V}_\sigma := \{ \mathbf{u} \in \mathbf{H}^1(\Omega): \nabla \cdot \mathbf{u} = 0 \text{ in } \Omega, \: \mathbf{u} = 0 \text{ on } \partial\Omega\},
\end{align*}
endowing the former with the same Hilbert structure as $\mathbf{L}^2(\Omega)$, whereas for the latter
	\[
	(\mathbf{u}, \mathbf{v})_{\mathbf{V}_\sigma}:=\int_\Omega \nabla \mathbf{u} : \nabla \mathbf{v}\,\mathrm{d} x, \qquad \| \mathbf{u} \|_{\mathbf{V}_\sigma}:= (\nabla \mathbf{u}, \nabla \mathbf{u})^\frac12.
	\]
We report the well known Korn's inequality
$$
\|\nabla \uu\|\leq \sqrt{2}\|D\uu\|,\quad \forall\, \uu\in\mathbf{V}_\sigma,
$$
where $D\uu=\frac{1}{2}(\nabla \uu+ \nabla^\mathrm{T}\uu)$. Next, we introduce the Stokes operator $\mathbf{A}:  \mathbf{V}_\sigma \to \mathbf{V}_\sigma^*$, which is the Riesz isomorphism between $\mathbf{V}_\sigma$ and its topological dual $\mathbf{V}_\sigma^*$, that is,	
    \[
\langle \mathbf{A}\uu, \mathbf{v}\rangle_{\mathbf{V}_\sigma^*,\mathbf{V}_\sigma } = \int_\Omega \nabla \uu : \nabla \mathbf{v}\,\mathrm{d}x.
    \]
 The inverse of $\mathbf{A}$ is denoted by $\mathbf{A}^{-1}$. In a similar fashion to what has been carried out for the operator $\mathcal{A}_N$, we can define the equivalent norm
$\| \mathbf{u}\|_\sharp := \| \nabla \mathbf{A}^{-1}\mathbf u \|$ in $\mathbf{V}_\sigma^*$. Besides, it holds the chain rule
\begin{equation*}
\langle\partial_t\mathbf{f}(t),\mathbf{A}^{-1}\mathbf{f}(t)\rangle_{\mathbf{V}_\sigma^*,\mathbf{V}_\sigma } =\frac{1}{2}\frac{\mathrm{d}}{\mathrm{d}t}\|\nabla\mathbf{A}^{-1}\mathbf{f}\|^2,\quad   \textrm{for a.a.}\ t \in (0,T),\ \ \forall\,\mathbf{f}\in H^1(0,T;\mathbf{V}_\sigma^*).\nonumber
\end{equation*}
Let us introduce the space $\mathbf{W}_\sigma := \mathbf{H}^2(\Omega) \cap \mathbf{V}_\sigma$. One can check that the norm $\|\uu\|_{\textbf{W}_\sigma} := \|\mathbf{A}\uu\|$ is equivalent to the standard $\mathbf{H}^2$-norm in $\mathbf{W}_\sigma$ (see \cite{G}). Moreover, we recall the following regularity result for the Stokes operator (see e.g.,  \cite[Appendix B]{GMT18}):
\begin{lemma}\label{stokes}
Let $\Omega\subset\mathbb{R}^3$ be a bounded domain with smooth boundary. For any $\mathbf{f} \in \mathbf{H}_{\sigma}$,
there exists a unique $\mathbf{u}\in \mathbf{W}_{\sigma}$ and $P\in H^1(\Omega)\cap L_0^2(\Omega)$ such that $-\Delta \mathbf{u}+\nabla P=\mathbf{f}$ a.e. in $\Omega$, that is, $\mathbf{u}=\mathbf{A}^{-1}\mathbf{f}$. Furthermore, we have
\begin{align*}
&\|\mathbf{u}\|_{\mathbf{H}^2}+\|\nabla P\|\le C\|\mathbf{f}\|,
\\
& \|P\|\le C \|\mathbf{f}\|^\frac12\|\nabla \mathbf{A}^{-1}\mathbf{f}\|^\frac12,
 \end{align*}
where $C$ is a positive constant that may depend on $\Omega$ but is independent of $\mathbf{f}$.
\end{lemma}	
\noindent

\subsection{Statement of main results} \label{sec:main}
\setcounter{equation}{0}

The energy of a biomembrane with fluids both inside and outside is
composed of the kinetic energy of the fluid and the elastic free energy.
Thus, let us define the total energy of problem \eqref{NS}--\eqref{IC} by
\begin{align}
\mathcal{E}(\uu(t),\phi(t)) = \frac12\|\uu(t)\|^2+ E(\phi(t)),\quad \forall\, t\geq0,
\label{totalE}
\end{align}
where $E(\phi)$ is given by \eqref{defiE}. Throughout this paper, we assume that $\eta\in \mathbb{R}$ is a given constant. Although the sign of $\eta$ does not play a role in the subsequent analysis, it is important from the point of view of modeling and application. When $\eta=0$, the energy $E(\phi)$ reduces to the approximate Willmore functional $E_\varepsilon(\phi)$ for the Canham--Helfrich bending energy $E_\text{elastic}(\phi)$, while for $\eta>0$, $E(\phi)$ is related to the Willmore regularization of the Cahn--Hilliard energy $E_{\text{CH}}$, which was introduced in \cite{TLVW09} to investigate strong anisotropy effects (corresponding to some non-constant coefficient $\eta$) arising during the growth and coarsening of thin films.

Next, we impose the following basic assumptions on nonlinear functions:
\begin{itemize}
\item[(A1)] $\nu\in C^2(\mathbb{R})$. There exists $\nu_*>0$ such that $\nu(s)\geq \nu_*$ for all $s\in\mathbb{R}$.
\item[(A2)] $m\in C^2(\mathbb{R})$. There exists $m_*>0$ such that $m(s)\geq m_*$, for all $s\in\mathbb{R}$.
\item[(A3)] The functions $F$ and $f$ take the following form:
    $$
    F(s)=\frac14(s^2-1)^2,\quad f(s)=F'(s)=s^3-s,\quad \forall\, s\in \mathbb{R}.
    $$
\end{itemize}
\begin{remark}\label{rem:Ff}
Thanks to the Sobolev embedding theorem $H^2(\Omega)\hookrightarrow L^\infty(\Omega)$ in three dimensions, we do not need to assume upper bounds for $\nu$ and $m$. Besides, more general form of $F$ (and thus $f$) can be treated with minor modifications in the subsequent proofs, at least from the mathematical point of view. For instance, we can easily extend the results to the case with a general polynomial $F(s)=\sum_{j=1}^{2p} a_js^j$, where $a_{2p}>0$, $p\in \mathbb{N}$ and $p\geq 2$. See also \cite{DLP19} for some further assumptions on admissible potential functions.
\end{remark}
We introduce the notion of weak solution (or, the so-called finite energy solution).
\begin{definition}\label{def:solution}
For any initial data $(\uu_0, \phi_0)\in \mathbf{H}_\sigma \times H^2_N(\Omega)$, a set of functions $( \uu, \phi, \mu, \omega)$ is called
a global weak solution to problem \eqref{NS}--\eqref{IC} on $[0,+\infty)$, if
\bea
&&\uu \in L^{\infty}(0, +\infty; \mathbf{H}_\sigma ) \cap L^2(0, +\infty;
\mathbf{V}_\sigma)\cap W^{1,\frac43}_{\mathrm{loc}}(0,+\infty; \mathbf{V}_\sigma^*), \label{reg-u}\\
&&\phi \in L^\infty(0, +\infty; H^2_N(\Omega))\cap L^2_{\mathrm{loc}}(0, +\infty; H^5(\Omega)\cap H^4_N(\Omega))\cap H^1_{\mathrm{loc}}(0, +\infty; (H^1(\Omega))^*), \label{reg-phi}\\
&& \mu\in L^2_{\mathrm{loc}}(0,+\infty; H^1(\Omega)),\label{reg-mu}\\
&& \omega \in L^\infty(0,+\infty; L^2(\Omega))\cap L^2_{\mathrm{loc}}(0,+\infty; H^3(\Omega)\cap H^2_N(\Omega)),\label{reg-ome}
\eea
satisfying the weak formulation
\begin{align}
& \langle \partial_t \uu, \mathbf{v} \rangle_{\mathbf{V}_\sigma^*, \mathbf{V}_\sigma}
+((\uu\cdot\nabla)\uu,\mathbf{v})
+(2\nu(\phi)D \uu,\nabla \mathbf{v})
=(\mu\nabla \phi, \mathbf{v}),\qquad \forall\, \mathbf{v}\in \mathbf{V}_\sigma,\label{weak1}\\
& \langle \partial_t \phi, \psi \rangle_{(H^1)^*,H^1}
+ (\uu \cdot \nabla  \phi,\psi)
+ (m(\phi)\nabla \mu, \nabla \psi)=0,\qquad\qquad  \forall\, \psi\in H^1(\Omega),
\label{weak2}
\end{align}
for a.a. $t\in (0,+\infty)$,
\begin{align}
&\mu=-\Delta \omega +f'(\phi)\omega+\eta\omega,\qquad \text{a.e. in }\Omega\times (0,+\infty),\label{weak3}\\
&\omega=-\Delta \phi+f(\phi), \qquad\quad\ \qquad  \text{a.e. in }\Omega\times (0,+\infty),\label{weak4}
\end{align}
the initial conditions \eqref{IC}, as well as the strong energy inequality
\be \mathcal{E}(\uu(t),\phi(t)) +\int_s^{t}\int_\Omega \big(2\nu(\phi(\tau))|D \uu(\tau)|^2 + m(\phi(\tau))|\nabla \mu(\tau)|^2\big)\,\mathrm{d}x\mathrm{d}\tau \leq
\mathcal{E}(\uu(s),\phi(s)),
\label{lowene}
\ee
for almost all $s\geq 0$, including $s=0$ and all $t\geq s$, where $\mathcal{E}(\uu(t),\phi(t))$ is given by \eqref{totalE}.
\end{definition}
\begin{remark}
\label{rem:reg1}
In Definition \ref{def:solution}, the initial data are attained in the following sense. Thanks to the regularity of $\phi$ and its time derivative, we infer that for any $T>0$, $\phi \in C([0,T];H^2(\Omega))$. Analogously, from the regularity of $\uu$ and its time derivative, it follows that $\uu \in C_w([0,T]; \mathbf{H}_\sigma)$, where the subscript $w$ means the weak continuity (in time).
\end{remark}

First, we prove the existence of global weak solutions.

\begin{theorem}[Existence]\label{thm:weak}
Assume that $\Omega\subset \mathbb{R}^3$ is a bounded domain of class $C^5$ and (A1)--(A3) are satisfied. For any initial datum $(\uu_0, \phi_0)\in \mathbf{H}_\sigma \times H^2_N(\Omega)$, problem \eqref{NS}--\eqref{IC} admits at least one global weak solution $(\uu, \phi,\mu,\omega)$ in the sense of Definition \ref{def:solution}.
\end{theorem}

Next, we derive some sufficient conditions under which a weak solution is unique in the class of weak solutions.

\begin{theorem}[Conditional uniqueness]\label{thm:weakuni}
Assume that $\Omega\subset \mathbb{R}^3$ is a bounded domain of class $C^5$ and (A1)--(A3) are satisfied. For arbitrary but fixed $T\in (0,+\infty)$, let $(\uu_i,\phi_i)$, $i=1,2$ be two weak solutions to problem \eqref{NS}--\eqref{IC} on $[0,T]$ corresponding to the same initial datum $(\uu_0, \phi_0)\in \mathbf{H}_\sigma \times H^2_N(\Omega)$. If one of the following conditions holds for $\uu_1$:
\begin{itemize}
\item[$\mathrm{(a)}$]  $\uu_1\in L^q(0,T;\mathbf{L}^p(\Omega))$, with $\displaystyle{\frac{2}{q}+\frac{3}{p}=1}$, $3<p\leq +\infty$;
\item[$\mathrm{(b)}$] $\nabla \uu_1\in L^q(0,T;\mathbf{L}^p(\Omega))$ with $\displaystyle{\frac{2}{q}+\frac{3}{p}=2}$, $\frac{3}{2}<p\leq +\infty$;
\item[$\mathrm{(c)}$] $\uu_1\in L^q(0,T;\mathbf{W}^{s,p}(\Omega))$ with $\displaystyle{\frac{2}{q}+\frac{3}{p}=1+s}$, $1<p,q< +\infty$, $s\geq 0$;
\end{itemize}
then $(\uu_1,\phi_1)=(\uu_2,\phi_2)$ a.e. in $\Omega\times [0,T]$.
\end{theorem}
\begin{remark}
There is a huge literature on the uniqueness of Leray--Hopf weak solutions to the Navier--Stokes equations for incompressible viscous fluids. Here, we try to verify some classical uniqueness criteria for our fluid-phase-field coupled system \eqref{NS}--\eqref{pot2} in a three dimensional bounded domain. One may refer to \cite{Se62} for condition $\mathrm{(a)}$, to \cite{Be02} for condition $\mathrm{(b)}$ and to \cite{Ri02} for condition $\mathrm{(c)}$. The results presented in Theorems \ref{thm:weakuni} indicate that in spite of the nonlinear coupling between the equations for the velocity field $\uu$ and the phase function $\phi$, the velocity field turns out to play a dominant role in the uniqueness property for weak solutions to problem \eqref{NS}--\eqref{IC} (cf. \cite{Zhao14}, where a fourth order phase-field equation for $\phi$ was considered). Since the higher-order parabolic equation yields some higher spatial regularity for $\phi$, possible difficulties from the convection term $\uu\cdot \nabla \phi$ as well as the Korteweg force term $\mu\nabla \phi$ are weakened in certain sense, comparing with the Navier--Stokes--Allen--Cahn system \cite{Wu17,GGW21} and the Navier--Stokes--Cahn--Hilliard system \cite{Ab,B99,GG09,GMT18,ZWH09}.
\end{remark}

Let us now introduce the notion of strong solution.

\begin{definition}\label{def:str}
Let $T\in (0,+\infty)$. For any initial datum $(\uu_0, \phi_0)\in \mathbf{V}_\sigma \times (H^5(\Omega)\cap H^4_N(\Omega))$, a set of functions $(\uu, \phi, \mu, \omega, P)$ is called a strong solution to problem \eqref{NS}--\eqref{IC} on $[0,T]$, if it admits the additional regularity
\bea
&&\uu \in L^{\infty}(0, T; \mathbf{V}_\sigma ) \cap L^2(0, T;
\mathbf{W}_\sigma) \cap H^1(0,T; \mathbf{H}_\sigma),
\label{reg-us}\\
&&\phi \in L^\infty(0,T; H^5(\Omega)\cap H^4_N(\Omega)) \cap L^2(0,T; H^6(\Omega)), \label{reg-phis}\\
&&\partial_t \phi \in L^\infty(0,T; (H^1(\Omega))^*)\cap L^2(0,T;H^2_N(\Omega)), \label{reg-phisb}\\
&& \mu\in L^\infty(0,T;H^1(\Omega))\cap L^2(0,T;H^2_N(\Omega)),
\label{reg-mus}\\
&& \omega \in L^\infty(0,T;H^3(\Omega)\cap H^2_N(\Omega))\cap L^2(0,T;H^4(\Omega)),\\
&& P\in L^\infty(0,T; H^1(\Omega)),
\label{reg-omes}
\eea
and satisfies the system \eqref{NS}--\eqref{pot2} a.e. in $\Omega\times (0,T)$ as well as the initial conditions \eqref{IC}.
\end{definition}

We can prove local well-posedness of problem \eqref{NS}--\eqref{IC} for arbitrarily large regular initial data in $\mathbf{V}_\sigma\times (H^5(\Omega)\cap H^4_N(\Omega))$.
\begin{theorem}[Local well-posedness]\label{thm:locstr}
Assume that $\Omega\subset \mathbb{R}^3$ is a bounded domain of class $C^6$ and (A1)--(A3) are satisfied. For any initial datum $(\uu_0, \phi_0)\in \mathbf{V}_\sigma\times (H^5(\Omega)\cap H^4_N(\Omega))$, there exists a time $T_0\in (0,+\infty)$ such that problem \eqref{NS}--\eqref{IC} admits a unique local strong solution $(\uu, \phi,\mu,\omega,P)$ on $[0,T_0]$ in the sense of Definition of \ref{def:str}.
\end{theorem}

Due to the weak coupling structure of problem \eqref{NS}--\eqref{IC}, we can derive some blow-up criteria that only involve the velocity field or its gradient.
\begin{theorem}[Blow-up criteria]\label{thm:buc}
Assume that $\Omega\subset \mathbb{R}^3$ is a bounded domain of class $C^6$ and (A1)--(A3) are satisfied. Let $(\uu,\phi,\mu,\omega,P)$ be a local strong solution to problem \eqref{NS}--\eqref{IC} on $[0,T_1)$ for some $T_1\in (0,+\infty)$. Suppose that one of the following condition holds,
\begin{itemize}
\item[$\mathrm{(a)}$] $ \displaystyle{\int_{0}^{T_1}\|\uu(t)\|_{\mathbf{L}^p}^q \,\mathrm{d}t < +\infty}$, with $\displaystyle{\frac{3}{p}+\frac{2}{q} = 1}$,  $3<p\leq +\infty$;
\item[$\mathrm{(b)}$] $ \displaystyle{\int_0^{T_1} \|\nabla \uu(t)\|_{\mathbf{L}^p}^q\,\mathrm{d}t < +\infty}$, with $\displaystyle{\frac{3}{p}+\frac{2}{q} = 2}$, $\displaystyle{\frac{3}{2}<p\leq 3}$;
\end{itemize}
 then $(\uu,\phi,\mu,\omega,P)$ can be extended beyond of the maximal time of existence $T_1$.
\end{theorem}

\section{Global Weak Solutions}\label{sec:weak}
\setcounter{equation}{0}

An important property of the coupled system \eqref{NS}--\eqref{IC} is that it satisfies a basic energy law, which implies the dissipative nature of the evolution problem. It states that the sum of the kinetic and elastic energy is dissipated due to viscosity of the fluid and certain diffusive process.

\begin{lemma}[Basic energy law]
\label{BEL}
Let $(\uu, \phi,\mu,\omega)$ be a smooth solution to the problem \eqref{NS}--\eqref{IC}.
The following energy identity holds:
 \be
 \dfrac{\mathrm{d}}{\mathrm{d}t}\Big(\frac{1}{2}\|\uu(t)\|^2+E(\phi(t))\Big)
+\int_\Omega \big(2\nu(\phi(t))|D\uu(t)|^2 +m(\phi(t)) |\nabla \mu(t)|^2\big)\,\mathrm{d}x =0,\quad \forall\, t>0. \label{basic energy law}
 \ee
\end{lemma}
\begin{proof}
For smooth solutions, a formal derivation can be carried out by multiplying \eqref{NS} by $\uu$ and \eqref{phasefield} by $\mu$,
respectively, then integrating over $\Omega$ and adding the resultants together, we obtain
\begin{align}
\frac{1}{2}\frac{\mathrm{d}}{\mathrm{d}t}\|\uu\|^2+ \int_\Omega \mu \partial_t \phi\,\mathrm{d}x
+\int_\Omega \big(2\nu(\phi)|D\uu|^2 +m(\phi) |\nabla \mu|^2\big)\,\mathrm{d}x
=0,
\label{bel-u}
\end{align}
where we have used the facts
$$
\int_\Omega \nabla P\cdot\uu\,\mathrm{d}x =0,
$$
thank to \eqref{incom} and the no-slip boundary condition for $\uu$. Next, using the equations \eqref{pot1}, \eqref{pot2} and integration by parts, we see that
\begin{align}
\int_\Omega \mu \partial_t \phi \,\mathrm{d}x
&= \int_\Omega \Big[(-\Delta \omega +f'(\phi)\omega)\partial_t \phi +\eta \omega \partial_t \phi\Big]\,\mathrm{d}x\notag\\
& =\int_\Omega \Big[\big(-\Delta \partial_t \phi + f'(\phi)\partial_t \phi\big)\omega +\eta \big(\nabla \phi\cdot\nabla \partial_t \phi + f(\phi)\partial_t \phi\big)\Big]\,\mathrm{d}x\notag\\
& =\frac{\mathrm{d}}{\mathrm{d}t}E(\phi(t)),
\label{bel-p}
\end{align}
which together with \eqref{bel-u} yields the conclusion.
\end{proof}

\subsection{Existence} \label{sec:Galer}
Based on Lemma \ref{BEL}, we can employ a suitable Galerkin scheme (cf. \cite{B99,DLP19,DLL07}) to prove the existence of global weak solutions to problem \eqref{NS}--\eqref{IC}. \smallskip

\textbf{The Galerkin Approximation}.
To this end, we recall that the set of eigenfunctions of the operator $-\Delta$ subject to the homogeneous Neumann boundary condition denoted by $\{w_{j}\}_{j=1}^{\infty}$ forms an orthonormal basis of $L^{2}(\Omega)$ and is also an orthogonal basis of $H^{2}_{N}(\Omega)$. In particular, we can take $w_{1} = 1$, which corresponds to the first eigenvalue $\lambda_1=0$. For every $j>1$, $w_j$ cannot be a constant and $\overline{w_j}=0$,  whence $\lambda_j=\|\nabla w_i\|^2>0$. Analogously, we denote the set of eigenfunctions of the Stokes operator $\mathbf{A}$ by $\{ \mathbf{v}_j\}_{j=1}^\infty$, which forms an orthonormal basis of $\mathbf{H}_\sigma$ and is also an orthogonal basis of $\mathbf{W}_\sigma$. For any given positive integer $n$, we define the finite-dimensional subspaces $W_n := \mathrm{span} \{w_1, ..., w_n\}\subset H^2_N(\Omega)$, $\mathbf{V}_n := \mathrm{span}\{\mathbf{v}_1, ...,\mathbf{v}_n\}\subset \mathbf{W}_\sigma$, with corresponding orthogonal projections $\Pi_n: L^2(\Omega) \to W_n$ (with respect to the inner product in $L^2(\Omega)$) and $\mathrm{P}_n: \mathbf{H}_\sigma \to \mathbf{V}_n$ (with respect to the inner product in $\mathbf{H}_\sigma$). Then we look for functions of the form
\begin{align*}
\uu_{n}(x,t) := \sum_{j=1}^{n} a_{n,j}(t) \mathbf{v}_{j}(x),
\quad  \phi_{n}(x,t)  := \sum_{j=1}^{n} b_{n,j}(t) w_{j}(x),
\quad \mu_{n}(x,t)  := \sum_{j=1}^{n} c_{n,j}(t) w_{j}(x),
\end{align*}
which solve the following approximating problem:
\begin{equation} \label{eq:weak}
\begin{cases}
\left\langle \partial_t \mathbf{u}_{n}, \mathbf{v} \right\rangle_{\mathbf{V}^*_\sigma, \mathbf{V}_\sigma}
+ ((\mathbf{u}_{n} \cdot \nabla)\mathbf{u}_{n},\mathbf{v}) + (2\nu(\phi_{n}) D \mathbf{u}_{n}, \nabla \mathbf{v}) = (\mu_{n} \nabla \phi_{n}, \mathbf{v}), & \quad \forall \: \mathbf{v} \in \mathbf{V}_n,\\
\left\langle \partial_t\phi_{n}, w \right\rangle_{(H^1)^*,H^1} + (\uu_{n} \cdot \nabla \phi_{n},w) + (m(\phi_n)\nabla \mu_{n}, \nabla w) = 0, & \quad \forall \: w \in W_n,\\
(\mu_{n},w) = -(\omega_n,\Delta w) + (f'(\phi_{n})\omega_n + \eta \omega_n, w), & \quad \forall \: w \in W_n,\\
\qquad \text{with}\ \omega_{n} = -\Delta \phi_{n} + f(\phi_{n}), & \\
\uu^{n}(\cdot,0) = \mathrm{P}_n(\uu_0) =: \uu_{0}^n,\quad \phi^{n}(\cdot, 0) = \Pi_n(\phi_{0}) =: \phi_{0}^n, & \quad \text{in } \Omega,
\end{cases}
\end{equation}
for a.a. $t\in (0,T)$.

Inserting the expressions of the approximate solutions $(\uu_n,\phi_n,\mu_n)$ into the above weak formulation, we arrive at a system of ordinary differential equations in the unknowns $(a_{n,j}(t)$, $b_{n,j}(t)$, $c_{n,j}(t))$, $j=1,...,n$.
Recalling the assumptions (A1)--(A3), we can conclude the following result from the classical Cauchy--Lipschitz theorem for the system of ODEs:
\begin{proposition}\label{wellODE}
Assume that the assumptions in Theorem \ref{thm:weak} are satisfied and $T>0$.
For any positive integer $n$, there exists a time $T_n \in(0,T]$ such that the approximating problem \eqref{eq:weak} admits a unique local solution $(\uu_{n},\phi_{n}, \mu_{n}, \omega_{n})$ on $[0,T_n]$, which is given by the functions $\{a_{n,j}(t), b_{n,j}(t), c_{n,j}(t)\}_{j = 1}^n\subset C^1([0,T_n])$.
\end{proposition}

\textbf{Uniform a priori estimates}. Since the proof of Proposition \ref{wellODE} is standard, we omit it here. Below we derive some a priori estimates to show that for any $n$, $T_{n} = T$ and the approximate solutions $(\uu_n, \phi_{n}, \mu_{n}, \omega_{n})$ are uniformly bounded with respect to the approximating parameter $n$ in suitable function spaces.\smallskip

\textit{First estimate}. Testing the second equation in \eqref{eq:weak} by $1$, using integration by parts, we have
$$
\frac{\mathrm{d}}{\mathrm{d}t}\int_\Omega \phi_n(t)\,\mathrm{d}x=0,
$$
for any $t\in (0,T_n)$. Integrating with respect to time yields
\begin{align}
\int_\Omega \phi_n(t)\,\mathrm{d}x=\int_\Omega \phi_n(0)\,\mathrm{d}x=\int_\Omega \phi_0^n\,\mathrm{d}x=\int_\Omega \phi_0\,\mathrm{d}x,\quad \forall\, t\in [0,T_n].\label{app-e0}
\end{align}

\textit{Second estimate}. Similar to the proof of Lemma \ref{BEL}, in \eqref{eq:weak}, we take the test functions $\mathbf{v}=\uu_n$ in the first equation for $\uu_n$, $w=\mu_n$ in the second equation and $w=\partial_t\phi_n$ in the third equation, respectively. After integration by parts, we find that
\be
 \dfrac{\mathrm{d}}{\mathrm{d}t}\Big(\frac{1}{2}\|\uu_n\|^2+E(\phi_n)\Big)
+\int_\Omega \big(2\nu(\phi_n)|D\uu_n|^2 +m(\phi_n) |\nabla \mu_n|^2\big)\,\mathrm{d}x =0, \label{app-bel}
\ee
for any $t\in (0,T_n)$. In the derivation of \eqref{app-bel}, we have used the fact $\partial_{t}\phi_{n} \in W_{n}$ and
\begin{align*}
&(\Pi_{{n}}\Delta\omega_n, \partial_{t}\phi_{n}) = (\omega_n, \Delta\Pi_{{n}}(\partial_{t}\phi_{n}))= (\omega_n, \Delta\partial_{t}\phi_{n}),\\
&(\Pi_{{n}}(f'(\phi_{n})\omega_n), \partial_{t}\phi_{n}) = (f'(\phi_{n})\omega_n, \Pi_{{n}}(\partial_{t}\phi_{n})) = (f'(\phi_{n})\omega_n,\partial_{t}\phi_{n}),\\
&(\Pi_{{n}}f(\phi_n),\partial_{t}\phi_{n})=(f(\phi_n),\Pi_{{n}}(\partial_{t}\phi_{n})) =(f(\phi_n),\partial_{t}\phi_{n}).
\end{align*}
Integrating \eqref{app-bel} with respect to time, we obtain the following energy identity for approximate solutions
\begin{align}
&\frac{1}{2}\|\uu_n(t)\|^2+E(\phi_n(t))+\int_0^t \int_\Omega \big(2\nu(\phi_n(\tau))|D\uu_n(\tau)|^2 +m(\phi_n(\tau)) |\nabla \mu_n(\tau)|^2\big)\,\mathrm{d}x \mathrm{d}\tau\notag\\
&\quad = \frac{1}{2}\|\uu_n(0)\|^2+E(\phi_n(0)),\quad \forall\, t\in (0,T_n].
\label{app-e1}
\end{align}
Since by definition, $\uu^n_0\to \uu_0$ in $\mathbf{H}_\sigma$ and $\phi_0^n\to \phi_0$ in $H^2(\Omega)$ as $n\to +\infty$, there exists $n^*\in \mathbb{N}$ such that for all $n>n^*>0$,
$$
\|\uu^n_0\|\leq 1+\|\uu_0\|,\quad \|\phi_0^n\|_{H^2}\leq 1+\|\phi_0\|_{H^2}.
$$
Then we infer from the Sobolev embedding theorem $H^1(\Omega)\hookrightarrow L^p(\Omega)$, $p\in [1,6]$, the assumption (A3) and Young's inequality that the initial approximating energy can be bounded by
\begin{align}
\frac{1}{2}\|\uu_n(0)\|^2+E(\phi_n(0)) \leq C(1+\|\uu_0\|^2+ \|\phi_0\|_{H^2}^2+ \|\phi_0\|_{L^6}^6)\leq C(1+\|\uu_0\|^2+ \|\phi_0\|_{H^2}^6),
\label{app-e2}
\end{align}
where the positive constant $C$ depends on $\Omega$, $\eta$, but is independent of $n$.

Next, we show that the energy $E(\phi_n)$ is bounded from below by a uniform constant. The case $\eta\geq 0$ is trivial, since $F(s)\geq 0$ for $s\in \mathbb{R}$. For the case $\eta<0$, we infer from the following simple fact
$$
sf(s)-2F(s)=s^4-s^2-\frac12(s^2-1)^2= \frac12 s^4-\frac12,\quad \forall\, s\in \mathbb{R}
$$
that
\begin{align*}
 \eta \int_\Omega \Big[\frac{1}{2}|\nabla \phi_n|^2 +F(\phi_n)\Big] \,\mathrm{d}x
& = - \frac{\eta}{2}\int_\Omega \phi_n\Delta\phi_n\, \mathrm{d} x +\eta \int_\Omega F(\phi_n)\,\mathrm{d}x \notag\\
& = \frac{\eta}{2} \int_\Omega \phi_n (-\Delta \phi_n+f(\phi_n))\,\mathrm{d}x - \frac{\eta}{2} \int_\Omega (\phi_n f(\phi_n)-2F(\phi_n))\,\mathrm{d}x\notag \\
& \geq -\frac{1}{4} \int_\Omega (-\Delta \phi_n +f(\phi_n))^2\,\mathrm{d} x -\frac{\eta^2}{4}\int_\Omega \phi_n^2\,\mathrm{d}x
-\frac{\eta}{4}\int_\Omega (\phi_n^4-1)\,\mathrm{d}x \notag\\
&\geq  -\frac{1}{4} \int_\Omega (-\Delta \phi_n +f(\phi_n))^2\,\mathrm{d} x -\frac{\eta}{8}\int_\Omega \phi_n^4\,\mathrm{d}x
+\left(\frac{\eta}{4}+\frac{\eta^3}{8}\right)|\Omega|.
\end{align*}
Hence, for $\eta\in \mathbb{R}$, it holds
\begin{align}
\frac12\|\uu_n(t)\|^2+E(\phi_n(t))\geq \frac12\|\uu_n(t)\|^2+ \frac{1}{4} \int_\Omega (-\Delta \phi_n(t) +f(\phi_n(t)))^2\,\mathrm{d} x-C, \quad \forall\, t\in (0,T_n],\label{app-e3}
\end{align}
where $C>0$ is a constant only depending on $\eta$ and $\Omega$. We refer to \cite{DLP19} for calculations for a more general situation.

To recover an $H^2$-estimate for $\phi_n$, we consider the Neumann problem
\begin{equation}
\begin{cases}
-\Delta \phi_n + f(\phi_n) = \omega_n,\quad\ \text{in}\ \Omega,\\
\partial_\mathbf{n} \phi_n=0,\qquad\qquad \quad \, \text{on}\ \ \partial \Omega.
\end{cases}
\label{ellipA}
\end{equation}
Multiplying the equation by $\phi_n$, integrating over $\Omega$, we have
\begin{align*}
\|\nabla \phi_n\|^2 +\int_\Omega \phi_n^4\,\mathrm{d}x
&= \int_\Omega (\omega_n\phi_n+\phi_n^2)\,\mathrm{d}x\leq \frac12 \|\omega_n \|^2+  \frac{3}{2}\|\phi_n\|^2\\
&\leq \frac12 \|\omega_n\|^2+  \frac{1}{2}\int_\Omega \phi_n^4\,\mathrm{d}x +\frac{9}{8}|\Omega|.
\end{align*}
Then from the above estimate and the Cauchy--Schwarz inequality we deduce that
$$
\|\phi_n\|_{H^1}\leq C(1+\|\omega_n\|),
$$
which further implies
\begin{align}
\|\phi_n\|_{H^2} &\leq C(\|\omega_n-f(\phi_n)\|+\|\phi_n\|)\leq C(\|\omega_n\|+\|\phi_n\|_{L^6}^3+\|\phi_n\|)\notag\\
& \leq C(1+\|\omega_n\|^3),
\label{app-e4}
\end{align}
where $C>0$ only depends on $\Omega$.

Hence, collecting the estimates \eqref{app-e2}, \eqref{app-e3} and \eqref{app-e4}, using the assumptions (A1), (A2) and Korn's inequality, we deduce from \eqref{app-e1} that
\begin{align}
&\|\uu_n(t)\|^2+\|\phi_n(t)\|^2_{H^2} + \|\omega_n(t)\|^2 \notag\\
&\quad + \int_0^t \big(\nu_*\|\nabla \uu(\tau)\|^2+ m_*\|\nabla \mu_n(\tau)\|^2\big)\,\mathrm{d}\tau \leq C,\quad \forall\, t\in (0,T_n],
\label{app-e5}
\end{align}
where the constant $C>0$ depends on $\|\uu_0\|$, $\|\phi_0\|_{H^2}$, $\eta$ and $\Omega$, but is independent of $n$ and $t$.

\begin{remark}
As a consequence of the uniform estimate \eqref{app-e5}, we can extend the local solution $(\uu_n,\phi_n,\mu_n,\omega_n)$ to be a global one that is defined on
the full time interval $[0, T]$, i.e., $t_n = T$ for all $n>n^*$. Besides, the estimate \eqref{app-e5} holds on $[0,T]$ with $C$ being independent of $T$ and $n$.
\end{remark}

\textit{Third estimate}. Testing the third equation in \eqref{eq:weak} by $1$ and after integration by parts, we infer from \eqref{app-e5} that
\begin{align}
\left|\int_\Omega \mu_n(t)\,\mathrm{d} x\right|
&= \left|\int_\Omega \big( f'(\phi_n)\omega_n+ \eta\omega_n\big)\,\mathrm{d}x\right|\notag\\
&\leq \|f'(\phi_n)\|\|\omega_n\|+|\eta||\Omega|^\frac12 \|\omega_n\|
\leq C,\quad \forall\, t\in (0,T],
\label{app-e6}
\end{align}
where $C>0$ is a constant independent of $n$.
Then it follows from the Poincar\'{e}--Wirtinger inequality that
\begin{align}
\int_0^t \|\mu_n(\tau)\|_{H^1}^2\,\mathrm{d}\tau\leq C,\quad  \forall\, t\in (0,T],\label{app-e7}
\end{align}
where the constant $C>0$ depends on $\|\uu_0\|$, $\|\phi_0\|_{H^2}$, $\eta$, $\Omega$ and $T$, but is independent of $n$.
\smallskip

\textit{Fourth estimate}.
From the following inequality
\begin{align*}
\|\uu_n\cdot \nabla\phi_n\|\leq \|\uu_n\|_{\mathbf{L}^3}\|\nabla \phi_n\|_{\mathbf{L}^6} \leq C\|\uu_n\|^\frac12\|\nabla \uu_n\|^\frac12\|\phi_n\|_{H^2}
\end{align*}
and \eqref{app-e5}, we also obtain that
\begin{align*}
\int_0^t \|\uu_n(\tau)\cdot \nabla\phi_n(\tau)\|^2\,\mathrm{d}\tau \leq C,\quad \forall\, t\in (0,T].
\end{align*}
Besides, for any $\psi\in L^2(0,T;H^1(\Omega))$, it follows from \eqref{app-e5} and the assumption (A2) that
\begin{align*}
&\left|\int_0^T(m(\phi_n(t))\nabla \mu_n(t),\nabla \psi)\,\mathrm{d}t\right|\\
&\quad \leq \sup_{t\in[0,T]}\|\sqrt{m(\phi_n(t))}\|_{L^\infty} \|\sqrt{m(\phi_n)}\nabla \mu_n\|_{L^2(0,T;\mathbf{L}^2(\Omega))}
\|\psi\|_{L^2(0,T;H^1(\Omega))}\\
&\quad \leq C.
\end{align*}
By comparison in \eqref{eq:weak}, we then infer from the above estimates that
\begin{align}
\int_0^t \|\partial_t \phi_n(\tau) \|_{(H^1)^*}^2 \,\mathrm{d} \tau \leq C,\quad \forall\, t\in (0,T],
\label{app-e8}
\end{align}
where the constant $C>0$ depends on $\|\uu_0\|$, $\|\phi_0\|_{H^2}$, $\eta$, $\Omega$ and $T$, but is independent of $n$.
\smallskip

\textit{Fifth estimate}. We proceed to estimate the time derivative of $\uu_n$.
From \eqref{app-e5} and \eqref{app-e7}, it follows that for all $t\in (0,T]$,
\begin{align}
\int_0^t \|\mu_n(\tau)\nabla \phi_n(\tau)\|^2\,\mathrm{d}\tau & \leq C\int_0^t \|\mu_n(\tau)\|_{L^6}^2\|\nabla \phi_n(\tau)\|_{\mathbf{L}^3}^2\,\mathrm{d}\tau\notag\\
& \leq C \sup_{\tau\in [0,t]} \|\phi_n(\tau)\|_{H^2}^2 \int_0^t \|\mu_n(\tau)\|_{H^1}^2\,\mathrm{d}\tau \leq C. \label{app-e9}
\end{align}
 Besides, the Ladyzhenskaya inequality yields that for arbitrary $\mathbf{v} \in L^{4}(0,T;\mathbf{V}_\sigma)$, it holds
\begin{align*}
\left|\int_0^t ((\uu_n(\tau)\cdot\nabla) \uu_n(\tau), \mathbf{v}(\tau))\,\mathrm{d}\tau\right|
& =\left| \int_{0}^{t} \int_{\Omega} (\uu_{n} (\tau)\otimes \uu_{n}(\tau)): \nabla \mathbf{v}(\tau)\,\mathrm{d}x\mathrm{d}\tau\right|\\
&\leq \int_{0}^{t} \|\uu_{n}(\tau)\|_{\mathbf{L}^{4}}^{2} \|\nabla \mathbf{v}(\tau)\|\,\mathrm{d}\tau \\
&\leq C \sup_{\tau\in[0,t]} \|\uu_{n}(\tau)\|^{\frac{1}{2}} \|\uu_{n}\|_{L^2(0,t;\mathbf{V}_\sigma)}^{\frac{3}{2}} \|\mathbf{v}\|_{L^{4}(0,t;\mathbf{V}_\sigma)},
\end{align*}
for all $t\in (0,T]$. Hence, we obtain
\begin{align}\label{app-e10}
\|\partial_{t} \uu_{n}\|_{L^{\frac{4}{3}}(0,t;\mathbf{V}^*_\sigma)} \leq C,\quad \forall\, t\in (0,T],
\end{align}
where the constant $C>0$ depends on $\|\uu_0\|$, $\|\phi_0\|_{H^2}$, $\eta$, $\Omega$ and $T$, but is independent of $n$. \smallskip

\textit{Sixth estimate}. We now derive some higher-order spatial estimates for $\phi_n$. Recall the definition that $\mu_n=\Pi_n\Delta^2 \phi_n+ \Pi_n L(\phi_n)=\Delta^2 \phi_n+ \Pi_n L(\phi_n)$, where the lower-order term is given by
\begin{align}
L(\phi_n):= -\Delta f(\phi_n)+ f'(\phi_n)\omega_n+\eta\omega_n, \label{LLL}
\end{align}
Then it follows from \eqref{app-e5} and the Sobolev embedding theorem that
$$
\|L(\phi_n(t))\|\leq \|\Delta f(\phi_n)\|+ \|f'(\phi_n)+\eta\|_{L^\infty}\|\omega_n\|\leq C,\quad \forall\, t\in [0,T],
$$
with $C>0$ depending on $\|\uu_0\|$, $\|\phi_0\|_{H^2}$, $\eta$, $\Omega$, but  independent of $n$ and $T$.
Consider the elliptic problem for $\phi_n$
\begin{equation}
\begin{cases}
\Delta^2 \phi_n = \mu_n-\Pi_n L(\phi_n),\quad\ \ \text{in}\ \Omega,\\
\partial_\mathbf{n} \phi_n=\partial_\mathbf{n}\Delta\phi_n=0,\qquad \quad \, \text{on}\ \ \partial \Omega.
\end{cases}
\label{ellipB}
\end{equation}
We deduce from the classical elliptic estimate for bi-Laplacian, \eqref{app-e5} and \eqref{app-e6} that
\begin{align*}
\|\phi_n\|_{H^4}&\leq C(\|\Delta^2 \phi_n\|+\|\phi_n\|) \\
&\leq C(\|\mu_n-\overline{\mu_n}\|+|\overline{\mu_n}|+\|\Pi_n L(\phi_n)\|+\|\phi_n\|)\\
&\leq C(\|\nabla\mu_n\|+\|L(\phi_n)\|+1)\\
&\leq C(\|\nabla\mu_n\|+1),
\end{align*}
which together with \eqref{app-e5} implies
\begin{align*}
\int_0^t \|\phi_n(\tau)\|_{H^4}^2\,\mathrm{d}\tau \leq C,\quad  \forall\, t\in (0,T].
\end{align*}
Since $\omega_n=-\Delta\phi_n+f(\phi_n)$, the above estimate and \eqref{app-e5} also yield
\begin{align*}
\int_0^t \|\omega_n(\tau)\|_{H^2}^2\,\mathrm{d}\tau \leq C,\quad  \forall\, t\in (0,T].
\end{align*}

\textbf{Proof of Theorem \ref{thm:weak}}. Collecting the above estimates, we see that for any given $T\in (0,+\infty)$,
\begin{align*}
        \uu_n & \text{ is  bounded in } L^\infty(0,T;\mathbf{H}_\sigma) \cap L^2(0,T;\mathbf{V}_\sigma) \cap W^{1, \frac{4}{3}}(0,T;\mathbf{V}_\sigma^*), \\
		\phi_n & \text{ is  bounded in } L^\infty(0,T;H^2(\Omega)) \cap L^2(0,T;H^4(\Omega))\cap H^1(0,T;(H^1(\Omega))^*), \\
		\mu_n & \text{ is bounded in } L^2(0,T;H^1(\Omega)),  \\
		\omega_n & \text{ is bounded in } L^\infty(0,T;L^2(\Omega))\cap L^2(0,T;H^2(\Omega)),
\end{align*}
uniformly with respect to the parameter $n$, provided that $n > n^*$. These estimates are sufficient for us to find a convergent subsequence of the approximate solutions $(\uu_n, \phi_{n}, \mu_{n}, \omega_{n})$ as $n\to+\infty$, whose limit denoted by $(\uu, \phi, \mu, \omega)$ satisfies the following regularity properties (and thus the boundary conditions)
\begin{align*}
    &    \uu \in L^\infty(0,T;\mathbf{H}_\sigma) \cap L^2(0,T;\mathbf{V}_\sigma) \cap W^{1, \frac{4}{3}}(0,T;\mathbf{V}_\sigma^*), \\
	&	\phi \in L^\infty(0,T;H^2_N(\Omega))\cap L^2(0,T;H^4_N(\Omega)) \cap H^1(0,T;(H^1(\Omega))^*), \\
	&	\mu\in L^2(0,T;H^1(\Omega)),  \\
	&	\omega \in  L^\infty(0,T;L^2(\Omega))\cap L^2(0,T;H^2_N(\Omega)),		
\end{align*}
such that
\begin{align*}
        & \uu_n \rightharpoonup \uu\ \ \text{ weakly star in } L^\infty(0,T;\mathbf{H}_\sigma)\ \text{and weakly in }  L^2(0,T;\mathbf{V}_\sigma),\\
        &\partial_t \uu_n \rightharpoonup \partial_t\uu\ \ \text{weakly in } L^\frac{4}{3}(0,T;\mathbf{V}_\sigma^*), \\
		&\phi_n \rightharpoonup \phi\ \  \text{ weakly star in }  L^\infty(0,T;H^2(\Omega)) \ \text{and weakly in }   L^2(0,T;H^4(\Omega)),\\
        &\partial_t\phi_n \rightharpoonup \partial_t\phi \ \ \text{weakly in }  L^2(0,T;(H^1(\Omega))^*), \\
		&\mu_n \rightharpoonup \mu \ \ \text{weakly in }  L^2(0,T;H^1(\Omega)),  \\
		&\omega_n\rightharpoonup \omega \ \ \text{weakly star in }  L^\infty(0,T;L^2(\Omega)) \ \text{and weakly in } L^2(0,T;H^2(\Omega)).
\end{align*}
Moreover, by the well-known Aubin--Lions lemma, we can conclude (again up to a subsequence)
\begin{align*}
& \uu_n \to \uu\ \ \text{strongly in } L^2(0,T;\mathbf{H}^{1-r}(\Omega)),\\
& \phi_n \to \phi\ \ \text{strongly in } C([0,T];H^{2-r}(\Omega))\cap    L^2(0,T;H^{4-r}(\Omega)),
\end{align*}
for any $r\in (0,\frac{1}{2})$. The above convergence results enable us to verify that the limit function $(\uu, \phi, \mu, \omega)$ satisfies the weak formulations \eqref{weak1}--\eqref{weak2} on $[0,T]$, and
\begin{align*}
&\mu = -\Delta \omega + f'(\phi)\omega + \eta \omega,\quad   \omega = -\Delta \phi +f(\phi),
\end{align*}
a.e. in $\Omega\times(0,T)$. Besides, the initial conditions are satisfied (using e.g., \cite[Lemma 3.1.7]{Zh}). To derive the strong energy inequality \eqref{lowene}, we can pass to the limit as $n\to+\infty$ in
\begin{align}
&\frac{1}{2}\|\uu_n(t)\|^2+E(\phi_n(t))+\int_0^t \int_\Omega \big(2\nu(\phi_n(\tau))|D\uu_n(\tau)|^2 +m(\phi_n(\tau)) |\nabla \mu_n(\tau)|^2\big)\,\mathrm{d}x \mathrm{d}\tau\notag\\
&\quad = \frac{1}{2}\|\uu_n(s)\|^2+E(\phi_n(s)),
\notag
\end{align}
for almost all $s\geq 0$ including $s=0$ and all $t\geq s$, by taking the available weak/strong subsequent convergence results into account and using the weak lower semicontinuity of norms. Since the compactness argument is standard (see, for instance, \cite{B99} for the Navier--Stokes--Cahn--Hilliard system, \cite{DLP19} for the functionalized Cahn--Hilliard equation and \cite{DLL07,CG19} for some different type of vesicle-fluid interaction models), we omit the details here.

We proceed to derive some higher-order spatial estimates for $\omega$ and $\phi$.
Consider the Neumann problem for $\omega$
$$
\begin{cases}
-\Delta \omega= \mu-f'(\phi)\omega-\eta\omega,\quad \text{in}\ \Omega,\\
\partial_\mathbf{n}\omega=0, \qquad\qquad\qquad\qquad\  \text{on}\ \partial\Omega.
\end{cases}
$$
We infer from the elliptic estimate that
\begin{align*}
\|\omega\|_{H^3}& \leq C(\| \mu-f'(\phi)\omega-\eta\omega\|_{H^1}+\|\omega\|)\notag\\
& \leq C\big(\|\mu\|_{H^1}+(\|f'(\phi)\|_{L^\infty}+1)\|\omega\|_{H^1}+ \|f''(\phi)\|_{L^\infty}\|\nabla \phi\|_{\mathbf{L}^3}\|\omega\|_{L^6}\big).
\end{align*}
As a consequence, it holds
\begin{align}
\int_0^t \|\omega(\tau)\|_{H^3}^2\,\mathrm{d}\tau \leq C,\quad  \forall\, t\in (0,T],\label{app-e11}
\end{align}
where the constant $C>0$ depends on $\|\uu_0\|$, $\|\phi_0\|_{H^2}$, $\eta$, $\Omega$ and $T$. Applying the elliptic estimate for $\phi$, which satisfies the Neumann problem $-\Delta \phi=w-f(\phi)$ in $\Omega$ and $\partial_\mathbf{n}\phi=0$ on $\partial\Omega$, we find that
\begin{align*}
\|\phi\|_{H^5}
&\leq C(\|\omega-f(\phi)\|_{H^3}+\|\phi\|)\\
&\leq C\|\omega\|_{H^3} +C(\|\phi\|_{L^\infty}^2\|\phi\|_{H^3}+\|\phi\|_{H^3}) \\
&\leq C(\|\omega\|_{H^3}+\|\phi\|_{H^3}),
\end{align*}
where we have used the estimate on product of functions in three dimensions. The above estimate combined with \eqref{app-e11} and the fact $\phi\in L^2(0,T;H^4(\Omega))$ yield that
\begin{align}
\int_0^t \|\phi(\tau)\|_{H^5}^2\,\mathrm{d}\tau \leq C,\quad  \forall\, t\in (0,T],\label{app-e12}
\end{align}
where the constant $C>0$ depends on $\|\uu_0\|$, $\|\phi_0\|_{H^2}$, $\eta$, $\Omega$ and $T$.

The proof of Theorem \ref{thm:weak} is complete. \hfill $\square$

\subsection{Conditional uniqueness}
In this subsection, we prove Theorem \ref{thm:weakuni} on the uniqueness of weak solutions. Let us consider two weak solutions $(\uu_i,\phi_i,\mu_i,\omega_i)$, $i=1,2$, to problem  \eqref{NS}--\eqref{IC} defined on $[0,T]$, corresponding to the same initial data $(\uu_0,\phi_0)$. Set the difference of solutions by
$$
\widehat{\uu}=\uu_1-\uu_2,\quad \widehat{\phi}=\phi_1-\phi_2,\quad \widehat{\mu}=\mu_1-\mu_2,\quad \widehat{\omega}=\omega_1-\omega_2, \quad \widehat{P}=P_1-P_2,
$$
where $P_i$ is the pressure that vanishes in the weak formulation for $\uu_i$. Thanks to the property of mass conservation, we find that $\int_\Omega \widehat{\phi}(t)\,\mathrm{d}x=0 $ for $t\in [0,T]$. Hence, the following elliptic estimates hold
\begin{equation}
\|\widehat{\phi}\|_{H^2}\leq C\|\Delta\widehat{\phi}\|,\quad
\|\widehat{\phi}\|_{H^4}\leq C\|\Delta^2\widehat{\phi}\|,
\label{ellip}
\end{equation}
where the constant $C$ only depends on $\Omega$. As a consequence, the functional
$$
\mathcal{H}(t)=\frac12\|\widehat{\uu}(t)\|^2+\frac12\|\Delta\widehat{\phi}(t)\|^2
$$
indeed measures the distance between solutions $(\uu_1,\phi_2)$ and $(\uu_2,\phi_2)$ in $\mathbf{L}^2(\Omega)\times H^2(\Omega)$.

For the sake of convenience, we formally write down the system satisfied by the difference $(\widehat{\uu}, \widehat{\phi},\widehat{\mu},\widehat{\omega}, \widehat{P})$ in the strong form:
\begin{align}
  &\partial_t \widehat{\uu}+(\widehat{\uu}\cdot\nabla) \uu_1
  + (\uu_2\cdot\nabla)\widehat{\uu}
  -\nabla \cdot (2\nu(\phi_1)D \widehat{\uu}) +\nabla \widehat{P} \notag \\
  &\quad = \nabla \cdot (2(\nu(\phi_1)-\nu(\phi_2))D \uu_2) + \widehat{\mu}\nabla \phi_1 + \mu_2\nabla\widehat{\phi},
  \label{NSd}\\
  &\nabla\cdot \widehat{\uu} =0, \label{incomd}\\
  &\widehat{\phi}_t+\widehat{\uu}\cdot\nabla\phi_1
  + \uu_2\cdot\nabla \widehat{\phi}
  = \nabla \cdot(m(\phi_1)\nabla \widehat{\mu})
  + \nabla \cdot((m(\phi_1)-m(\phi_2))\nabla \mu_2),  \label{phasefieldd}\\
  &\widehat{\mu}=- \Delta \widehat{\omega} + f'(\phi_1) \widehat{\omega} + (f'(\phi_1)-f'(\phi_2))\omega_2  + \eta \widehat{\omega}, \label{pot1d}\\
  &\widehat{\omega}= - \Delta \widehat{\phi} + f(\phi_1)-f(\phi_2), \label{pot2d}
\end{align}
in $\Omega\times(0,T)$, subject to the boundary and initial conditions
\begin{align}
&\widehat{\uu}=\mathbf{0},\quad {\partial}_{\mathbf{n}}\widehat{\phi}
  ={\partial}_{\mathbf{n}}\Delta\widehat{\phi}= {\partial}_{\mathbf{n}}\widehat{\mu}=0,\qquad\qquad &\textrm{on}& \   \partial\Omega\times(0,T),\label{boundaryd}\\
  &\uu|_{t=0}=\mathbf{0},\ \ \phi|_{t=0}=0,\qquad &\textrm{in}&\ \Omega.
\label{ICd}
\end{align}
As a preliminary step, using the expressions \eqref{pot1d} and \eqref{pot2d}, we reformulate the difference of chemical potentials $\widehat{\mu}$ as follows
\begin{align*}
\widehat{\mu} & =\Delta^2 \widehat{\phi} - \Delta(f(\phi_1)-f(\phi_2)) - (f'(\phi_1)+\eta) \Delta\widehat{\phi} + (f'(\phi_1)+\eta)(f(\phi_1)-f(\phi_2))\\
&\quad + (f'(\phi_1)-f'(\phi_2))\omega_2,\\
&=: \Delta^2 \widehat{\phi} +G(\phi_1,\phi_2).
\end{align*}
Using the uniform bounds for weak solutions $\phi_1$, $\phi_2$ in $L^\infty(0,T;H^2(\Omega))$, we derive an estimate on the nonlinear term $G(\phi_1,\phi_2)$.
\begin{lemma}\label{G}
The nonlinear term $G(\phi_1,\phi_2)$ satisfies the following estimates:
\begin{align}
&\|G(\phi_1,\phi_2)\|\leq C\|\Delta\widehat{\phi}\|,\label{GG1}\\
&\|\nabla G(\phi_1,\phi_2)\|\leq C\|\nabla \Delta \widehat{\phi}\|+C(1+\|\phi_1\|_{H^3}+\|\phi_2\|_{H^3})\|\Delta \widehat{\phi}\|,\label{GG2}
\end{align}
for a.a. $t\in [0,T]$, where the constant $C>0$ depends on $\|\uu_0\|$, $\|\phi_0\|_{H^2}$, $\eta$ and $\Omega$.
\end{lemma}
\begin{proof}
By the definition of $G(\phi_1,\phi_2)$, we get
\begin{align}
\|G(\phi_1,\phi_2)\|&\leq \|\Delta(f(\phi_1)-f(\phi_2))\| +\|(f'(\phi_1)+\eta) \Delta\widehat{\phi}\| + \|(f'(\phi_1)+\eta)(f(\phi_1)-f(\phi_2))\|\notag \\
&\quad + \|(f'(\phi_1)-f'(\phi_2))\omega_2\|.\notag
\end{align}
Using the expression of $f$ and the $L_t^\infty H^2_x$-estimates for $\phi_1$, $\phi_2$, the terms on the right-hand side can be estimated by
\begin{align*}
&\|\Delta(f(\phi_1)-f(\phi_2))\|\notag\\
&\quad \leq \|3\phi_1^2\Delta\phi_1-3\phi_2^2\Delta\phi_2\|
+\|6\phi_1|\nabla \phi_1|^2- 6\phi_2|\nabla \phi_2|^2\|+\|\Delta\widehat{\phi}\|\notag\\
&\quad \leq C\|\phi_1\|^2_{L^\infty}\|\Delta \widehat{\phi}\|+C(\|\phi_1\|_{L^\infty}+\|\phi_2\|_{L^\infty})
\|\Delta\phi_2\|\|\widehat{\phi}\|_{L^\infty}\notag\\
&\qquad +C\|\nabla \phi_1\|_{\mathbf{L}^4}^2\|\widehat{\phi}\|_{L^\infty} + C\|\phi_2\|_{L^\infty}(\|\nabla \phi_1\|_{\mathbf{L}^3}+\|\nabla \phi_2\|_{\mathbf{L}^3})\|\nabla \widehat{\phi}\|_{\mathbf{L}^6}+\|\Delta\widehat{\phi}\|\notag\\
&\quad \leq C\|\Delta \widehat{\phi}\|,
\end{align*}
\begin{align*}
\|(f'(\phi_1)+\eta) \Delta\widehat{\phi}\|\leq \|f'(\phi_1)+\eta\|_{L^\infty}\|\Delta\widehat{\phi}\|\leq C\|\Delta\widehat{\phi}\|,
\end{align*}
and
\begin{align*}
&\|(f'(\phi_1)+\eta)(f(\phi_1)-f(\phi_2))\|+  \|(f'(\phi_1)-f'(\phi_2))\omega_2\|\notag\\
&\quad \leq \|f'(\phi_1)+\eta\|_{L^\infty}\|\phi_1^2+\phi_1\phi_2+\phi_2^2+1\|_{L^\infty}
\|\widehat{\phi}\|+ 3\|\phi_1+\phi_2\|_{L^\infty}\|\omega_2\|\|\widehat{\phi}\|_{L^\infty}\notag\\
&\quad \leq C\|\Delta\widehat{\phi}\|.
\end{align*}
Collecting the above estimates, we conclude that
\begin{align}
\|G(\phi_1,\phi_2)\|\leq C\|\Delta\widehat{\phi}\|.\notag
\end{align}
Next, we estimate the $L^2$-norm of $\nabla G(\phi_1,\phi_2)$ in the following way
\begin{align*}
\|\nabla G(\phi_1,\phi_2)\|
&\leq \|\nabla \Delta(f(\phi_1)-f(\phi_2))\| +\|f''(\phi_1)\nabla \phi_1 \Delta\widehat{\phi}\|
+\|(f'(\phi_1)+\eta)\nabla \Delta\widehat{\phi}\| \notag\\
&\quad +\|f''(\phi_1)\nabla \phi_1 (f(\phi_1)-f(\phi_2))\|
+\|(f'(\phi_1)+\eta)\nabla (f(\phi_1)-f(\phi_2))\|\notag\\
&\quad +\|\nabla(f'(\phi_1)-f'(\phi_2))\omega_2\|
+\|(f'(\phi_1)-f'(\phi_2))\nabla \omega_2\|.
\end{align*}
The terms on the right-hand side can be estimated as follows
\begin{align}
& \|\nabla \Delta(f(\phi_1)-f(\phi_2))\|\notag\\
& \quad \leq
\|(\phi_1^2+\phi_1\phi_2+\phi_2^2)\widehat{\phi}\|_{H^3}
+\|\nabla \Delta \widehat{\phi}\|\notag\\
&\quad \leq C \|\phi_1^2+\phi_1\phi_2+\phi_2^2\|_{H^3}\|\widehat{\phi}\|_{L^\infty}
+C\|\phi_1^2+\phi_1\phi_2+\phi_2^2\|_{L^\infty}\|\nabla \Delta \widehat{\phi}\|+ \|\nabla \Delta \widehat{\phi}\|\notag\\
&\quad \leq C\|\nabla \Delta \widehat{\phi}\|+ C(\|\phi_1\|_{H^3}+\|\phi_2\|_{H^3})\|\Delta \widehat{\phi}\|,\notag
\end{align}
\begin{align*}
&\|f''(\phi_1)\nabla \phi_1 \Delta\widehat{\phi}\|\leq \|f''(\phi_1)\|_{L^\infty}\|\nabla \phi_1\|_{\mathbf{L}^3} \|\Delta\widehat{\phi}\|_{L^6}
\leq C(\|\nabla \Delta\widehat{\phi}\|+\|\Delta\widehat{\phi}\|),
\end{align*}
$$
\|(f'(\phi_1)+\eta)\nabla \Delta\widehat{\phi}\| \leq \|f'(\phi_1)+\eta\|_{L^\infty}\|\nabla \Delta\widehat{\phi}\|\leq C\|\nabla \Delta\widehat{\phi}\|,
$$
\begin{align*}
&\|f''(\phi_1)\nabla \phi_1 (f(\phi_1)-f(\phi_2))\|
+\|(f'(\phi_1)+\eta)\nabla (f(\phi_1)-f(\phi_2))\|\notag\\
&\quad \leq \|f''(\phi_1)\|_{L^\infty}\|\nabla \phi_1\|_{\mathbf{L}^3}\|f(\phi_1)-f(\phi_2)\|_{L^6}
+ \|f'(\phi_1)+\eta\|_{L^\infty}\|\nabla (f(\phi_1)-f(\phi_2))\|\notag\\
&\quad \leq  C\|\Delta\widehat{\phi}\|,
\end{align*}
\begin{align*}
\|\nabla(f'(\phi_1)-f'(\phi_2))\omega_2\|&\leq
\|\nabla(f'(\phi_1)-f'(\phi_2))\|_{\mathbf{L}^3}\|\omega_2\|_{L^6}\\
&\leq C(\|\nabla \widehat{\phi}\|_{\mathbf{L}^3}+\|\widehat{\phi}\|_{L^\infty})\|\omega_2\|_{H^1}\\
&\leq C\|\omega_2\|_{H^1}\|\Delta\widehat{\phi}\|,
\end{align*}
and
\begin{align*}
\|(f'(\phi_1)-f'(\phi_2))\nabla \omega_2\|\leq \|f'(\phi_1)-f'(\phi_2)\|_{L^\infty}\|\nabla \omega_2\|
\leq C\|\nabla \omega_2\|\|\Delta\widehat{\phi}\|.
\end{align*}
Collecting the above estimates, we can deduce that
$$
\|\nabla G(\phi_1,\phi_2)\|\leq C\|\nabla \Delta \widehat{\phi}\|+C(1+\|\phi_1\|_{H^3}+\|\phi_2\|_{H^3}+\|\omega_2\|_{H^1})\|\Delta \widehat{\phi}\|.
$$
Noticing that
$$
\|\omega_2\|_{H^1}\leq \|\Delta \phi_2\|_{H^1}+\|f(\phi_2)\|_{H^1}\leq C(\|\phi_2\|_{H^3}+1),
$$
we then conclude \eqref{GG2}.

The proof is complete.
\end{proof}

\textbf{Proof of Theorem \ref{thm:weakuni}}.
Since $(\uu_1,\phi_1)$ and $(\uu_2,\phi_2)$ are two weak solutions in the sense of Definition \ref{def:solution}, both of them satisfy the energy inequality:
\be \mathcal{E}(\uu_i(t),\phi_i(t)) +\int_0^{t}\int_\Omega \left(2\nu(\phi_i)|D \uu_i(\tau)|^2 + m(\phi_i)|\nabla \mu_i(\tau)|^2\right)\,\mathrm{d}x\mathrm{d}\tau \leq
\mathcal{E}(\uu_0,\phi_0),
\label{lowenea}
\ee
for any $t\in [0,T]$, $i=1,2$.

\begin{lemma}\label{regut}
Let one of the conditions $\mathrm{(a)}$, $\mathrm{(b)}$ or $\mathrm{(c)}$ be satisfied.
We introduce the spaces
\begin{align*}
&\mathbf{X}=L^\infty(0,T;\mathbf{H}_\sigma)\cap L^2(0,T;\mathbf{V}_\sigma),\\
&\mathbf{Y}_{\mathrm{a}}= L^q(0,T;\mathbf{L}^p(\Omega)),\\
&\mathbf{Y}_{\mathrm{b}}= L^q(0,T;\mathbf{W}^{1,p}(\Omega)),\\
&\mathbf{Y}_{\mathrm{c}}= L^q(0,T;\mathbf{W}^{s,p}(\Omega)),\\
&\mathbf{Z}_{\mathrm{k}}=\mathbf{X} \cap \mathbf{Y}_{\mathrm{k}},\quad \mathrm{k}\in \{\mathrm{a}, \mathrm{b}, \mathrm{c}\}.
\end{align*}
with corresponding choices of parameters. Then for $\mathrm{k}\in \{\mathrm{a}, \mathrm{b}, \mathrm{c}\}$, we have
\begin{align*}
& \partial_t\uu_1\in \mathbf{X}^*= (L^\infty(0,T;\mathbf{H}_\sigma))^* + L^2(0,T;\mathbf{V}_\sigma^*),\\
& \partial_t\uu_2\in \mathbf{Z}_{\mathrm{k}}^*= (L^\infty(0,T;\mathbf{H}_\sigma))^* + L^2(0,T;\mathbf{V}_\sigma^*)+ \mathbf{Y}_\mathrm{k}^*.
\end{align*}
\end{lemma}
\begin{proof}
We first observe that
\begin{align*}
\left|\int_0^T\! \int_\Omega \mu_i(t)\nabla \phi_i(t)\cdot \mathbf{v}(t)\,\mathrm{d} x\mathrm{d}t\right|
&\leq \int_0^T\|\mu_i(t)\|_{L^3}\|\nabla \phi_i(t)\|\|\mathbf{v}(t)\|_{\mathbf{L}^6}\,\mathrm{d}t\\
&\leq \|\phi_i\|_{L^\infty(0,T;H^1(\Omega))}\|\mu_i\|_{L^2(0,T;H^1(\Omega))}\|\mathbf{v}\|_{L^2(0,T;\mathbf{V}_\sigma)},
\end{align*}
for any function $\mathbf{v} \in L^2(0,T;\mathbf{V}_\sigma)$, $i=1,2$. Keeping the above estimate for Korteweg forces in mind, the conclusion under condition $\mathrm{(c)}$ is a direct consequence of \cite[Lemma 1]{Ri02}. Concerning condition $\mathrm{(a)}$, the conclusion is an easy consequence of H\"{o}lder's inequality and the interpolation inequality $\|\uu_1\|_{\mathbf{L}^\frac{2p}{p-2}}\leq C\|\nabla \uu_1\|^{\frac{3}{p}} \|\uu_1\|^{1-\frac{3}{p}}$ (see \cite{Se62}). For the sake of completeness, we sketch the proof here.
For any $\mathbf{v} \in \mathbf{X}$, it follows that
\begin{align*}
&\left| \int_0^T\! \int_\Omega (\uu_1(t) \cdot \nabla) \uu_1 (t) \cdot \mathbf{v}(t)\,\mathrm{d}x \mathrm{d}t \right|\\
&\quad = \left|\int_0^T\int_\Omega (\uu_1(t)\cdot\nabla ) \mathbf{v}(t) \cdot \uu_1(t)\,\mathrm{d}x \mathrm{d}t \right| \\
&\quad \leq \int_0^T \|\uu_1(t)\|_{\mathbf{L}^\frac{2p}{p-2}}\|\uu_1(t)\|_{\mathbf{L}^p} \|\nabla \mathbf{v}(t)\|\,\mathrm{d}t\\
&\quad \leq C\int_0^T\|\nabla \uu_1(t)\|^{\frac{3}{p}} \|\uu_1(t)\|^{1-\frac{3}{p}}\|\uu_1(t)\|_{\mathbf{L}^p} \|\nabla \mathbf{v}(t)\| \,\mathrm{d}t\\
&\quad \leq C\|\uu_1\|_{L^\infty(0,T;\mathbf{H}_\sigma)}^{1-\frac{3}{p}} \|\nabla \uu_1\|_{L^2(0,T;\mathbf{L}^2(\Omega))}^\frac{3}{p}
\|\uu_1\|_{L^q(0,T;\mathbf{L}^p(\Omega))}\|\nabla \mathbf{v}\|_{L^2(0,T;\mathbf{L}^2(\Omega))},
\end{align*}
which yields that $\partial_t\uu_1\in \mathbf{X}^*$ by comparison in the weak formulation of $\uu_1$. In a similar manner, for any $\mathbf{v}\in \mathbf{Z}_{\mathrm{a}}$, we have
\begin{align*}
\left| \int_0^T \int_\Omega (\uu_2(t) \cdot \nabla) \uu_2(t) \cdot \mathbf{v}(t)\,\mathrm{d}x \mathrm{d}t \right|
& \leq \int_0^T \|\uu_2(t)\|_{\mathbf{L}^\frac{2p}{p-2}} \|\nabla \mathbf{u}_2(t)\|\|\mathbf{v}(t)\|_{\mathbf{L}^p}\,\mathrm{d}t\\
& \leq C\int_0^T\|\nabla \uu_2(t)\|^{1+\frac{3}{p}} \|\uu_2(t)\|^{1-\frac{3}{p}} \|\mathbf{v}(t)\|_{\mathbf{L}^p}  \,\mathrm{d}t\\
& \leq C\|\nabla \uu_2\|_{L^2(0,T;\mathbf{L}^2(\Omega))}^{1+\frac{3}{p}} \|\uu_2\|_{L^\infty(0,T;\mathbf{H}_\sigma)}^{1-\frac{3}{p}} \|\mathbf{v}\|_{L^q(0,T;\mathbf{L}^p(\Omega))},
\end{align*}
which then implies $\partial_t\uu_2\in \mathbf{Z}_{\mathrm{a}}^*$. Next, concerning condition $\mathrm{(b)}$, from the interpolation inequality $\|\uu\|_{\mathbf{L}^\frac{2p}{p-1}}\leq C\|\nabla \uu\|^{\frac{3}{2p}} \|\uu\|^{1-\frac{3}{2p}}$ we deduce that
\begin{align*}
&\left| \int_0^T\! \int_\Omega (\uu_1(t) \cdot \nabla) \uu_1 (t) \cdot \mathbf{v}(t)\,\mathrm{d}x \mathrm{d}t \right|\\
&\quad \leq \int_0^T \|\nabla \uu_1(t)\|_{\mathbf{L}^p} \|\uu_1(t)\|_{\mathbf{L}^\frac{2p}{p-1}} \| \mathbf{v}(t)\|_{\mathbf{L}^\frac{2p}{p-1}}\,\mathrm{d}t\\
&\quad \leq C\int_0^T \|\nabla \uu_1(t)\|_{\mathbf{L}^p}\|\nabla \uu_1(t)\|^{\frac{3}{2p}} \|\uu_1(t)\|^{1-\frac{3}{2p}}\|\nabla \mathbf{v}(t)\|^{\frac{3}{2p}} \|\mathbf{v}(t)\|^{1-\frac{3}{2p}}\,\mathrm{d}t\\
&\quad \leq C \|\uu_1\|_{L^\infty(0,T;\mathbf{H}_\sigma)}^{1-\frac{3}{2p}}\|\mathbf{v}\|_{L^\infty(0,T;\mathbf{H}_\sigma)}^{1-\frac{3}{2p}}
\|\nabla \uu_1\|^{\frac{3}{2p}}_{L^2(0,T;\mathbf{L}^2(\Omega))}\|\nabla \mathbf{v}\|^{\frac{3}{2p}}_{L^2(0,T;\mathbf{L}^2(\Omega))}\|\nabla \uu_1\|_{L^q(0,T;\mathbf{L}^p(\Omega))},
\end{align*}
for any $\mathbf{v} \in \mathbf{X}$, and
\begin{align*}
\left| \int_0^T \int_\Omega (\uu_2(t) \cdot \nabla) \uu_2(t) \cdot \mathbf{v}(t)\,\mathrm{d}x \mathrm{d}t \right|
& = \left| \int_0^T \int_\Omega (\uu_2(t) \cdot \nabla) \mathbf{v}(t) \cdot  \uu_2(t) \,\mathrm{d}x \mathrm{d}t \right|\\
& \leq \int_0^T \|\uu_2(t)\|^2_{\mathbf{L}^\frac{2p}{p-1}} \|\nabla \mathbf{v}(t)\|_{\mathbf{L}^p}\,\mathrm{d}t\\
& \leq C\int_0^T \|\nabla \uu_2(t)\|^{\frac{3}{p}} \|\uu_2(t)\|^{2-\frac{3}{p}} \|\nabla \mathbf{v}(t)\|_{\mathbf{L}^p}  \,\mathrm{d}t\\
& \leq C \|\uu_2\|_{L^\infty(0,T;\mathbf{H}_\sigma)}^{2-\frac{3}{p}} \|\nabla \uu_2\|_{L^2(0,T;\mathbf{L}^2(\Omega))}^{\frac{3}{p}} \|\nabla \mathbf{v}\|_{L^q(0,T;\mathbf{L}^p(\Omega))},
\end{align*}
for any $\mathbf{v}\in \mathbf{Z}_{\mathrm{b}}$. These estimates imply that $\partial_t\uu_1\in \mathbf{X}^*$ and  $\partial_t\uu_2\in \mathbf{Z}_{\mathrm{b}}^*$, respectively.
\end{proof}

Thanks to Lemma \ref{regut}, we can take $\uu_2$ as a test function in the weak formulation of $\uu_1$ and vice versa. Adding the resultants together and integrating with respect to time over any interval $[0,t]\subset [0,T]$, we obtain
\begin{align}
&(\uu_1(t),\uu_2(t)) - \|\uu_0\|^2\notag \\
&\quad =-\int_0^t\big((\uu_1(\tau) \cdot \nabla) \uu_1(\tau),\uu_2(\tau)\big)\,\mathrm{d} \tau
-\int_0^t\big((\uu_2(\tau) \cdot \nabla) \uu_2(\tau),\uu_1(\tau)\big)\,\mathrm{d} \tau \notag \\
&\qquad -\int_0^t\int_\Omega 2\nu(\phi_1(\tau))D\uu_1(\tau): \nabla \uu_2(\tau)\,\mathrm{d}x\mathrm{d}\tau
-\int_0^t\int_\Omega 2\nu(\phi_2(\tau))D\uu_2(\tau): \nabla \uu_1(\tau)\,\mathrm{d}x\mathrm{d}\tau\notag \\
&\qquad + \int_0^t\int_\Omega \mu_1(\tau)\nabla \phi_1(\tau) \cdot \uu_2(\tau) \,\mathrm{d}x\mathrm{d}\tau
+ \int_0^t\int_\Omega \mu_2(\tau)\nabla \phi_2(\tau) \cdot \uu_1(\tau) \,\mathrm{d}x\mathrm{d}\tau.
\label{diff-u1}
\end{align}
The procedure to derive the identity \eqref{diff-u1} can be made rigorous by using the mollification argument in \cite{Se62}, see also \cite{Ga00,Ri02}, namely, approximating $\uu_i$ by smooth solenoidal functions and then passing to the limit. Adding the energy inequalities \eqref{lowenea} for $i=1,2$ together and then subtracting the identity \eqref{diff-u1}, we arrive at
\begin{align}
&\frac12\|\uu(t)\|^2 + E(\phi_1(t))+ E(\phi_2(t)) + \int_0^t\int_\Omega 2\nu(\phi_1(\tau)) |D\widehat{\uu}(\tau)|^2\,\mathrm{d}x\mathrm{d}\tau  \notag\\
&\qquad + \int_0^{t} \int_\Omega \left(m(\phi_1)|\nabla \mu_1(\tau)|^2+ m(\phi_2)|\nabla \mu_2(\tau)|^2\right)\,\mathrm{d}x\mathrm{d}\tau  \notag\\
&\quad \leq  2E(\phi_0)-\int_0^t\int_\Omega (\widehat{\uu}(\tau)\cdot\nabla ) \uu_1(\tau) \cdot \widehat{\uu}(\tau)\,\mathrm{d}x\mathrm{d}\tau
\notag \\
&\qquad -\int_0^t\int_\Omega (2(\nu(\phi_1(\tau))-\nu(\phi_2(\tau)))D \uu_2(\tau)): D \widehat{\uu}(\tau)\,\mathrm{d}x \mathrm{d}\tau
 \notag\\
&\qquad - \int_0^t\int_\Omega \mu_1(\tau)\nabla \phi_1(\tau) \cdot \uu_2(\tau) \,\mathrm{d}x\mathrm{d}\tau
- \int_0^t\int_\Omega \mu_2(\tau)\nabla \phi_2(\tau) \cdot \uu_1(\tau) \,\mathrm{d}x\mathrm{d}\tau,
\label{diff-u2}
\end{align}
where we have used the incompressibility condition for $\uu_i$ and the following fact due to integration by parts
\begin{align*}
&\int_0^t\big((\uu_1(\tau) \cdot \nabla) \uu_1(\tau),\uu_2(\tau)\big)\,\mathrm{d} \tau
+\int_0^t\big((\uu_2(\tau) \cdot \nabla) \uu_2(\tau),\uu_1(\tau)\big)\,\mathrm{d} \tau\\
&\quad = -\int_0^t\int_\Omega (\uu_2(\tau)\cdot\nabla) \widehat{\uu}(\tau)\cdot \widehat{\uu}(\tau)\,\mathrm{d}x\mathrm{d}\tau
-\int_0^t\int_\Omega (\widehat{\uu}(\tau)\cdot\nabla ) \uu_1(\tau) \cdot \widehat{\uu}(\tau)\,\mathrm{d}x\mathrm{d}\tau\\
&\quad =-\int_0^t\int_\Omega (\widehat{\uu}(\tau)\cdot\nabla ) \uu_1(\tau) \cdot \widehat{\uu}(\tau)\,\mathrm{d}x\mathrm{d}\tau.
\end{align*}
Next, from the regularity properties of $(\uu_i, \phi_i, \mu_i, \omega_i)$, we are allowed to take $\mu_i\in L^2(0,T;H^1(\Omega))$ as a test function in the equation for $\phi_i$.  Integrating the resultant over $\Omega\times (0,t)$ and after integration by parts, we obtain the energy identities
\begin{align}
E(\phi_i(t)) +\int_0^{t} \int_\Omega m(\phi_i(\tau))|\nabla \mu_i(\tau)|^2 \,\mathrm{d}x\mathrm{d}\tau
= E(\phi_0)- \int_0^{t} \int_\Omega (\uu_i(\tau)\cdot\nabla \phi_i(\tau))\mu_i(\tau) \,\mathrm{d}x\mathrm{d}\tau,
\label{diff-u3}
\end{align}
for $i=1,2$. Subtracting \eqref{diff-u3} from \eqref{diff-u2} leads to
\begin{align}
&\frac12\|\uu(t)\|^2 + \int_0^t\int_\Omega 2\nu(\phi_1(\tau)) |D\widehat{\uu}(\tau)|^2\,\mathrm{d}x\mathrm{d}\tau  \notag\\
&\quad \leq -\int_0^t\int_\Omega (\widehat{\uu}(\tau)\cdot\nabla ) \uu_1(\tau) \cdot \widehat{\uu}(\tau)\,\mathrm{d}x\mathrm{d}\tau
 \notag\\
&\qquad
-\int_0^t\int_\Omega (2(\nu(\phi_1(\tau))-\nu(\phi_2(\tau)))D \uu_2(\tau)): D \widehat{\uu}(\tau)\,\mathrm{d}x \mathrm{d}\tau\notag \\
&\qquad +\int_0^{t} \int_\Omega \widehat{\mu}(\tau)\nabla\phi_1(\tau) \cdot \widehat{\uu}(\tau)\,\mathrm{d}x\mathrm{d}\tau
+\int_0^{t} \int_\Omega \mu_2(\tau)\nabla\widehat{\phi}(\tau) \cdot \widehat{\uu}(\tau)\,\mathrm{d}x\mathrm{d}\tau\notag\\
&\quad = -\int_0^t\int_\Omega (\widehat{\uu}(\tau)\cdot\nabla ) \uu_1(\tau) \cdot \widehat{\uu}(\tau)\,\mathrm{d}x\mathrm{d}\tau\notag\\
&\qquad -\int_0^t\int_\Omega (2(\nu(\phi_1(\tau))-\nu(\phi_2(\tau)))D \uu_2(\tau)): D \widehat{\uu}(\tau)\,\mathrm{d}x \mathrm{d}\tau\notag \\
&\qquad +\int_0^{t} \int_\Omega \mu_2(\tau)\nabla\widehat{\phi}(\tau) \cdot \widehat{\uu}(\tau)\,\mathrm{d}x\mathrm{d}\tau
+\int_0^{t} \int_\Omega \Delta^2\widehat{\phi}(\tau) \nabla \phi_1(\tau)\cdot \widehat{\uu}(\tau)\,\mathrm{d}x\mathrm{d}\tau
 \notag\\
&\qquad
+ \int_0^{t}\int_\Omega G(\phi_1(\tau),\phi_2(\tau)) \nabla \phi_1(\tau)\cdot \widehat{\uu}(\tau)\,\mathrm{d}x\mathrm{d}\tau.
\label{diff-u4}
\end{align}
Besides, we are allowed to take $\Delta^2\widehat{\phi}\in L^2(0,T;H^1(\Omega))$ as a test function for \eqref{phasefieldd}. Integrating the resultant over $\Omega\times (0,t)$ and after integration by parts, we find that
\begin{align}
&\frac12\|\Delta\widehat{\phi}(t)\|^2 + \int_0^t\int_\Omega m(\phi_1(\tau))|\nabla \Delta^2 \widehat{\phi}(\tau)|^2\,\mathrm{d}x\mathrm{d}\tau \notag \\
&\quad = -\int_0^t\int_\Omega \big(\widehat{\uu}(\tau)\cdot\nabla\phi_1(\tau)
  + \uu_2(\tau)\cdot\nabla \widehat{\phi}(\tau) \big) \Delta^2\widehat{\phi}(\tau)\,\mathrm{d} x\mathrm{d}\tau\notag \\
&\qquad   -\int_0^t\int_\Omega m(\phi_1(\tau))\nabla G(\phi_1(\tau),\phi_2(\tau)) \cdot \nabla \Delta^2\widehat{\phi}(\tau)\,\mathrm{d} x\mathrm{d}\tau\notag\\
&\qquad   -\int_0^t\int_\Omega (m(\phi_1(\tau))-m(\phi_2(\tau)))\nabla \mu_2(\tau)\cdot \nabla \Delta^2\widehat{\phi}(\tau)\,\mathrm{d} x\mathrm{d}\tau.
  \label{dif-2}
\end{align}
Adding \eqref{diff-u4} and \eqref{dif-2} together, using the assumptions (A1), (A2) and Korn's inequality, we
obtain the following inequality
\begin{align}
&\frac12\|\uu(t)\|^2+\frac12\|\Delta\widehat{\phi}(t)\|^2 + \nu_*\int_0^t\|\nabla \widehat{\uu}(\tau)\|^2\,\mathrm{d}\tau
+ m_*\int_0^t \|\nabla \Delta^2 \widehat{\phi}(\tau)\|^2\,\mathrm{d}\tau \notag \\
& \quad \leq -\int_0^t\int_\Omega (\widehat{\uu}(\tau)\cdot\nabla ) \uu_1(\tau) \cdot \widehat{\uu}(\tau)\,\mathrm{d}x\mathrm{d}\tau\notag\\
& \qquad -\int_0^t\int_\Omega (2(\nu(\phi_1(\tau))-\nu(\phi_2(\tau)))D \uu_2(\tau)): D \widehat{\uu}(\tau)\,\mathrm{d}x \mathrm{d}\tau\notag \\
& \qquad +\int_0^{t} \int_\Omega \mu_2(\tau)\nabla\widehat{\phi}(\tau) \cdot \widehat{\uu}(\tau)\,\mathrm{d}x\mathrm{d}\tau
+ \int_0^{t}\int_\Omega G(\phi_1(\tau),\phi_2(\tau)) \nabla \phi_1(\tau)\cdot \widehat{\uu}(\tau)\,\mathrm{d}x\mathrm{d}\tau\notag\\
& \qquad -\int_0^t\int_\Omega  \uu_2(\tau)\cdot\nabla \widehat{\phi}(\tau) \Delta^2\widehat{\phi}(\tau)\,\mathrm{d} x\mathrm{d}\tau\notag \\
&\qquad   -\int_0^t\int_\Omega m(\phi_1(\tau))\nabla G(\phi_1(\tau),\phi_2(\tau)) \cdot \nabla \Delta^2\widehat{\phi}(\tau)\,\mathrm{d} x\mathrm{d}\tau\notag\\
&\qquad   -\int_0^t\int_\Omega (m(\phi_1(\tau))-m(\phi_2(\tau)))\nabla \mu_2(\tau)\cdot \nabla \Delta^2\widehat{\phi}(\tau)\,\mathrm{d} x\mathrm{d}\tau\notag\\
  &\quad =: \sum_{j=1}^7 I_j.
  \label{dif-3}
\end{align}

Let us first estimate $J_2,...J_7$ on the right-hand side of \eqref{dif-3}, leaving the term $J_1$ to the end.
Using the assumption (A1) and the bounds for $\phi_1$, $\phi_2$, we see that
\begin{align}
J_2&\leq C\int_0^t\|\nu(\phi_1(\tau))-\nu(\phi_2(\tau))\|_{L^\infty}\|D \uu_2(\tau)\|\|\nabla  \widehat{\uu}(\tau)\|\,\mathrm{d}\tau\notag\\
&\leq C\int_0^t\left\|\int_0^1 \nu'(s\phi_1(\tau)+(1-s)\phi_2(\tau))\widehat{\phi}(\tau)\,\mathrm{d}s\right\|_{L^\infty}
\|D \uu_2(\tau)\|\|\nabla  \widehat{\uu}(\tau)\|\,\mathrm{d}\tau \notag\\
&\leq C\int_0^t\|\Delta\widehat{\phi}(\tau)\| \|D \uu_2(\tau)\|\|\nabla  \widehat{\uu}(\tau)\|\,\mathrm{d}\tau \notag\\
&\leq \frac{\nu_*}{6}\int_0^t \|\nabla  \widehat{\uu}(\tau)\|^2\,\mathrm{d}\tau
+C \int_0^t\|\nabla \uu_2(\tau)\|^2\|\Delta\widehat{\phi}(\tau)\|^2\,\mathrm{d}\tau.
\notag
\end{align}
In a similar manner, we have
\begin{align}
J_7
&\leq \int_0^t  \|m(\phi_1(\tau))-m(\phi_2(\tau))\|_{L^\infty}\|\nabla \mu_2(\tau)\|\|\nabla \Delta^2\widehat{\phi}(\tau)\|\,\mathrm{d}\tau \notag\\
&\leq \int_0^t \left\|\int_0^1 m'(s\phi_1(\tau)+(1-s)\phi_2(\tau))\widehat{\phi}(\tau)\,\mathrm{d}s \right\|_{L^\infty} \|\nabla \mu_2(\tau)\| \|\nabla \Delta^2\widehat{\phi}(\tau)\|\mathrm{d}\tau \notag\\
&\leq C \int_0^t\|\Delta\widehat{\phi}(\tau)\| \|\nabla \mu_2(\tau)\|\|\nabla \Delta^2\widehat{\phi}(\tau)\|\mathrm{d}\tau\notag\\
&\leq \frac{m_*}{6} \int_0^t \|\nabla \Delta^2\widehat{\phi}(\tau)\|^2\,\mathrm{d}\tau
+C \int_0^t \|\nabla \mu_2(\tau)\|^2\|\Delta\widehat{\phi}(\tau)\|^2\,\mathrm{d}\tau.\notag
\end{align}
For $J_3$, it follows that
\begin{align}
J_3 & \leq \int_0^t\|\mu_2(\tau)\|_{L^3}\|\nabla\widehat{\phi}(\tau)\|_{\mathbf{L}^6} \|\widehat{\uu}(\tau)\| \,\mathrm{d}\tau \notag  \\
& \leq \int_0^t\|\widehat{\uu}(\tau)\|^2\,\mathrm{d}\tau + C \int_0^t \|\mu_2(\tau)\|_{L^3}^2\|\Delta\widehat{\phi}(\tau)\|^2\,\mathrm{d}\tau.\notag
\end{align}
From \eqref{GG1}, we obtain
\begin{align}
J_4& \leq \int_0^t \|G(\phi_1(\tau),\phi_2(\tau))\| \|\nabla \phi_1(\tau)\|_{\mathbf{L}^6}\| \widehat{\uu}(\tau)\|_{\mathbf{L}^3}\,\mathrm{d}\tau \notag \\
&\leq C \sup_{\tau\in[0,t]}\|\phi_1(\tau)\|_{H^2} \int_0^t \|\Delta\widehat{\phi}(\tau)\|\| \widehat{\uu}(\tau)\|^\frac12
\|\nabla \widehat{\uu}(\tau)\|^\frac12\,\mathrm{d}\tau \notag\\
&\leq \frac{\nu_*}{6}\int_0^t\|\nabla  \widehat{\uu}(\tau)\|^2\,\mathrm{d}\tau
      +C\int_0^t \| \widehat{\uu}(\tau)\|^2\,\mathrm{d}\tau
      +C\int_0^t \|\Delta\widehat{\phi}(\tau)\|^2\,\mathrm{d}\tau.\notag
\end{align}
For $J_5$, using integration by parts, we get
\begin{align}
J_5
&= \int_0^t \int_\Omega (\widehat{\phi}(\tau) \uu_2(\tau))\cdot\nabla \Delta^2\widehat{\phi}(\tau)\,\mathrm{d} x\mathrm{d}\tau\notag\\
& \leq \int_0^t  \|\uu_2(\tau)\|_{\mathbf{L}^3} \| \widehat{\phi}(\tau)\|_{L^6}\|\nabla \Delta^2\widehat{\phi}(\tau)\|\,\mathrm{d}\tau \notag\\
&\leq \frac{m_*}{6} \int_0^t \|\nabla \Delta^2\widehat{\phi}(\tau)\|^2 \,\mathrm{d}\tau
+C \int_0^t\|\uu_2(\tau)\|_{\mathbf{L}^3}^2 \| \Delta \widehat{\phi}(\tau)\|^2\,\mathrm{d}\tau.\notag
\end{align}
Combining the elliptic estimate for the bi-Laplacian subject to Neumann boundary conditions
\begin{align*}
\|\widehat{\phi}\|_{H^5}
&\leq C(\|\Delta^2 \widehat{\phi}\|_{H^1}+\|\widehat{\phi}\|)\\
&\leq C \|\nabla \Delta^2 \widehat{\phi}\|
+ C\|\Delta^2 \widehat{\phi}\|,
\end{align*}
with the interpolation $\|\widehat{\phi}\|_{H^4}\leq C\|\widehat{\phi}\|_{H^2}^\frac13 \|\widehat{\phi}\|_{H^5}^\frac23$, and Young's inequality,
we find that
\begin{align*}
\|\widehat{\phi}\|_{H^5}\leq C(\|\nabla \Delta^2 \widehat{\phi}\|+\|\Delta \widehat{\phi}\|).
\end{align*}
The above estimate together with $\|\widehat{\phi}\|_{H^3}\leq C\|\widehat{\phi}\|_{H^2}^\frac23\|\widehat{\phi}\|_{H^5}^\frac13$ also implies
$$
\|\widehat{\phi}\|_{H^3}\leq C\|\nabla \Delta^2 \widehat{\phi}\|^\frac13\|\Delta \widehat{\phi}\|^\frac23+C\|\Delta \widehat{\phi}\|.
$$
Thus, for $J_6$,   we can deduce from \eqref{GG2} that
\begin{align*}
J_6
&\leq \int_0^t  \|m(\phi_1(\tau))\|_{L^\infty}\|\nabla G(\phi_1(\tau),\phi_2(\tau))\|\| \nabla \Delta^2\widehat{\phi}(\tau)\| \,\mathrm{d}\tau \\
&\leq C\int_0^t \big(\|\nabla \Delta \widehat{\phi}(\tau)\|+ (1+\|\phi_1(\tau)\|_{H^3}+\|\phi_2(\tau)\|_{H^3})\|\Delta \widehat{\phi}(\tau)\|\big) \| \nabla \Delta^2\widehat{\phi}(\tau)\|\,\mathrm{d}\tau \\
&\leq C\int_0^t \big(\|\nabla \Delta^2 \widehat{\phi}(\tau)\|^\frac43\|\Delta \widehat{\phi}(\tau)\|^\frac23 + \|\nabla \Delta^2 \widehat{\phi}(\tau)\|^\frac13\|\Delta \widehat{\phi}(\tau)\|^\frac53\big)\,\mathrm{d}\tau \\
&\quad + C \int_0^t \big(1+\|\phi_1(\tau)\|_{H^3}+\|\phi_2(\tau)\|_{H^3}\big)\|\Delta \widehat{\phi}(\tau)\|\| \nabla \Delta^2\widehat{\phi}(\tau)\|\,\mathrm{d}\tau
\\
&\leq \frac{m_*}{6} \int_0^t \|\nabla \Delta^2\widehat{\phi}(\tau)\|^2\,\mathrm{d}\tau
+C\int_0^t \big(1+\|\phi_1(\tau)\|_{H^3}^2+\|\phi_2(\tau)\|_{H^3}^2\big)\|\Delta \widehat{\phi}(\tau)\|^2\,\mathrm{d}\tau.
\end{align*}

To complete the proof of Theorem \ref{thm:weak}, it remains to estimate $J_1$. We distinguish three cases.

\textbf{Case 1}. Assume that condition $\mathrm{(a)}$ is satisfied. Applying the interpolation inequality $\|\widehat{\uu}\|_{L^\frac{2p}{p-2}}\leq C\|\nabla \widehat{\uu}\|^{\frac{3}{p}} \|\widehat{\uu}\|^{1-\frac{3}{p}}$ and Young's inequality, we can deduce that (see e.g., \cite{Se62,Zhao14})
\begin{align*}
J_1&= \int_0^t\int_\Omega (\widehat{\uu}(\tau)\cdot\nabla ) \widehat{\uu}(\tau) \cdot \uu_1(\tau)\,\mathrm{d}x\mathrm{d}\tau \\
&\leq \int_0^t \|\widehat{\uu}(\tau)\|_{\mathbf{L}^\frac{2p}{p-2}}\|\nabla \widehat{\uu}(\tau)\|\|\uu_1(\tau)\|_{\mathbf{L}^p}\,\mathrm{d}\tau \\
&\leq C\int_0^t \|\nabla \widehat{\uu}(\tau)\|^{1+\frac{3}{p}} \|\widehat{\uu}(\tau)\|^{1-\frac{3}{p}}\|\uu_1(\tau)\|_{\mathbf{L}^p}\,\mathrm{d}\tau\\
&\leq \frac{\nu_*}{6}\int_0^t\|\nabla  \widehat{\uu}(\tau)\|^2\,\mathrm{d}\tau
+C \int_0^t \|\uu_1(\tau)\|_{\mathbf{L}^p}^q \|\widehat{\uu}(\tau)\|^2\,\mathrm{d}\tau.
\end{align*}
Collecting the estimates for $J_1,...J_7$, we infer from \eqref{dif-3} that
\begin{align}
&\frac{1}{2}\int_0^t \big(\|\widehat{\uu}(\tau)\|^2 +\|\Delta\widehat{\phi}(\tau)\|^2\big)\,\mathrm{d}\tau
+ \frac{\nu_*}{2}\int_0^t \|\nabla \widehat{\uu}(\tau)\|^2\,\mathrm{d}\tau
+ \frac{m_*}{2}\int_0^t\|\nabla \Delta^2 \widehat{\phi}(\tau)\|^2\,\mathrm{d}\tau\notag\\
&\quad  \leq C\int_0^t h_{\mathrm{a}}(\tau)\big(\|\widehat{\uu}(\tau)\|^2 + \|\Delta\widehat{\phi}(\tau)\|^2\big)\,\mathrm{d}\tau,
\quad \forall\, t\in [0,T],\notag
\end{align}
where
$$
h_{\mathrm{a}}(t)=1+\|\uu_1(t)\|_{\mathbf{L}^p}^q + \|\nabla \uu_2(t)\|^2+\|\mu_2(t)\|_{H^1}^2 +\|\phi_1(t)\|_{H^3}^2+\|\phi_2(t)\|_{H^3}^2.
$$
Since $h_{\mathrm{a}}(t)\in L^1(0,T)$, as a direct consequence of Gronwall's lemma, we can conclude
$$
\|\widehat{\uu}(t)\|=\|\Delta\widehat{\phi}(t)\|=0,\quad \forall\,t\in [0,T],
$$
which yields the uniqueness result.

\textbf{Case 2}. Assume that condition $\mathrm{(b)}$ is satisfied. In this case, $J_1$ can be estimated by using the interpolation inequality  $\|\widehat{\uu}\|_{L^\frac{2p}{p-1}}\leq C\|\nabla \widehat{\uu}\|^\frac{3}{2p}\|\widehat{\uu}\|^{1-\frac{3}{2p}}$ and Young's inequality such that (cf. \cite{Be02})
\begin{align*}
J_1&\leq \int_0^t\|\nabla \uu_1(\tau)\|_{\mathbf{L}^p} \|\widehat{\uu}(\tau)\|_{\mathbf{L}^\frac{2p}{p-1}}^2\,\mathrm{d}\tau \\
&\leq C\int_0^t \|\nabla \uu_1(\tau)\|_{\mathbf{L}^p} \|\nabla \widehat{\uu}(\tau)\|^\frac{3}{p}\|\widehat{\uu}(\tau)\|^{2-\frac{3}{p}}\,\mathrm{d}\tau \\
&\leq \frac{\nu_*}{6}\int_0^t \|\nabla  \widehat{\uu}(\tau)\|^2\,\mathrm{d}\tau
+C \int_0^t \|\nabla \uu_1(\tau)\|_{\mathbf{L}^p}^q\|\widehat{\uu}(\tau)\|^2\,\mathrm{d}\tau.
\end{align*}
The above estimate together with the previous estimates for $J_2,...,J_7$ and \eqref{dif-3} yields that
\begin{align}
&\frac{1}{2}\int_0^t \big(\|\widehat{\uu}(\tau)\|^2 +\|\Delta\widehat{\phi}(\tau)\|^2\big)\,\mathrm{d}\tau
+ \frac{\nu_*}{2}\int_0^t \|\nabla \widehat{\uu}(\tau)\|^2\,\mathrm{d}\tau
+ \frac{m_*}{2}\int_0^t\|\nabla \Delta^2 \widehat{\phi}(\tau)\|^2\,\mathrm{d}\tau\notag\\
&\quad  \leq C\int_0^t h_{\mathrm{b}}(\tau)\big(\|\widehat{\uu}(\tau)\|^2 + \|\Delta\widehat{\phi}(\tau)\|^2\big)\,\mathrm{d}\tau,
\quad \forall\, t\in [0,T],\notag
\end{align}
where
$$
h_{\mathrm{b}}(t)=1+\|\nabla \uu_1\|_{\mathbf{L}^p}^q + \|\nabla \uu_2(t)\|^2+\|\mu_2(t)\|_{H^1}^2+\|\phi_1(t)\|_{H^3}^2+\|\phi_2(t)\|_{H^3}^2.
$$
From the fact $h_{\mathrm{b}}(t)\in L^1(0,T)$ and Gronwall's lemma, we obtain
$$
\|\widehat{\uu}(t)\|=\|\Delta\widehat{\phi}(t)\|=0,\quad \forall\,t\in [0,T].
$$

\textbf{Case 3}. Assume that condition $\mathrm{(c)}$ is satisfied. The term $J_1$ can be estimated as in \cite[Section 2]{Ri02} such that
\begin{align}
J_1&\leq C\int_0^t \|\uu(\tau)\|^{1+s-\frac{3}{p}}\|\nabla \uu(\tau)\|^{1-s+\frac{3}{p}}\|\uu_1(\tau)\|_{\mathbf{W}^{s,p}}\,\mathrm{d}\tau\notag\\
&\leq \frac{\nu_*}{6}\int_0^t \|\nabla  \widehat{\uu}(\tau)\|^2\,\mathrm{d}\tau
+C\int_0^t \|\uu_1(\tau)\|_{\mathbf{W}^{s,p}}^q\|\uu(\tau)\|^2\,\mathrm{d}\tau.\notag
\end{align}
Hence, it holds
\begin{align}
&\frac{1}{2}\int_0^t \big(\|\widehat{\uu}(\tau)\|^2 +\|\Delta\widehat{\phi}(\tau)\|^2\big)\,\mathrm{d}\tau
+ \frac{\nu_*}{2}\int_0^t \|\nabla \widehat{\uu}(\tau)\|^2\,\mathrm{d}\tau
+ \frac{m_*}{2}\int_0^t\|\nabla \Delta^2 \widehat{\phi}(\tau)\|^2\,\mathrm{d}\tau\notag\\
&\quad  \leq C\int_0^t h_{\mathrm{c}}(\tau)\big(\|\widehat{\uu}(\tau)\|^2 + \|\Delta\widehat{\phi}(\tau)\|^2\big)\,\mathrm{d}\tau,
\quad \forall\, t\in [0,T],\notag
\end{align}
where
$$
h_{\mathrm{c}}(t)=1+\|\uu_1(\tau)\|_{\mathbf{W}^{s,p}}^q + \|\nabla  \uu_2(t)\|^2+\|\mu_2(t)\|_{H^1}^2+\|\phi_1(t)\|_{H^3}^2+\|\phi_2(t)\|_{H^3}^2.
$$
From the fact $h_{\mathrm{c}}(t)\in L^1(0,T)$ and Gronwall's lemma, we again have
$$
\|\widehat{\uu}(t)\|=\|\Delta\widehat{\phi}(t)\|=0,\quad \forall\,t\in [0,T].
$$
The proof of Theorem \ref{thm:weakuni} is complete. \hfill $\square$

\section{Local Strong Solutions}\label{sec:str}
\setcounter{equation}{0}

In this section, we study the strong solutions to problem \eqref{NS}--\eqref{IC}.

\subsection{Local well-posedness}
To prove the existence result, we just need to derive sufficient higher-order uniform estimates for the Galerkin approximate solutions $(\uu_n,\phi_n,\mu_n,\omega_n)$ in Section \ref{sec:Galer} and then pass to the limit as $n\to+\infty$ to extract a subsequence converging in some more regular spaces.

From the arguments in Section \ref{sec:weak}, we have obtained the following global estimates
\begin{lemma}
\label{low}
For any initial datum $(\uu_0,\phi_0)\in
\mathbf{H}_\sigma\times H^2_N(\Omega)$, the approximate
solution $(\uu_n$, $\phi_n$, $\mu_n$, $\omega_n)$ to problem  \eqref{eq:weak} satisfies the following uniform-in-time estimates
 \begin{align}
 & \|\uu_n(t)\|+\|\phi_n(t)\|_{H^2} +\|\omega_n(t)\|\leq C, \quad \forall\, t\geq 0, \label{unilow}
 \\
 & \int_0^{+\infty} \left(\nu_*\|\nabla \uu_n(t)\|^2+m_*\|\nabla \mu_n(t)\|^2\right)\,\mathrm{d}t\leq C,
 \label{int}
 \end{align}
 where the constant $C>0$ depends on $\|\uu_0\|, \|\phi_0\|_{H^2}$, $\eta$ and $\Omega$, but is independent of $n$ and $t$.
 \end{lemma}
Next, we derive higher-order differential inequalities for $\uu_n$ and $\phi_n$ (in terms of $\mu_n$).

\begin{lemma}\label{lem:duu}
The approximate solution $(\uu_n,\phi_n,\mu_n)$ satisfies
\begin{align}
&\frac{1}{2}\frac{\mathrm{d}}{\mathrm{d}t} \|\nabla \uu_n\|^2
+\frac{\nu_*}{2}\|\mathbf{A}\uu_n\|^2 +\frac{3\nu_*}{16C_*}\|\partial_t\uu_n\|^2\notag\\
&\quad \leq C\|\nabla \uu_n\|^6 + C\|\nabla \uu_n\|^2 +C\|\nabla \mu_n\|^2,
\label{duu-A1}
\end{align}
where $C_*, C>0$ are some constants depending on $\|\uu_0\|$, $\|\phi_0\|_{H^2}$, $\eta$ and $\Omega$, but independent of $n$ and $t$.
\end{lemma}
\begin{proof}
The derivation of \eqref{duu-A1} is motivated by the argument in \cite{GMT18} for the Navier--Stokes--Cahn--Hilliard system. For the sake of completeness, we present the main steps here and point out some necessary modifications. Testing the first equation in \eqref{eq:weak} by $\mathbf{A}\uu_n$, after integration by parts, we get
\be \frac12 \frac{\mathrm{d}}{\mathrm{d}t} \|\nabla \uu_n\|^2-2(\nabla\cdot (\nu(\phi_n) D\uu_n), \mathbf{A}\uu_n)
 =
 -(\uu_n\cdot\nabla \uu_n, \mathbf{A}\uu_n) + (\mu_n\nabla \phi_n,\mathbf{A} \uu_n).
 \label{duu}
\ee
The right-hand side of \eqref{duu} can be estimated by using H\"{o}lder's inequality, the Sobolev embedding theorem, the Poincar\'e--Wirtinger inequality and the estimates \eqref{app-e6}, \eqref{unilow} such that
\begin{align}
|(\uu_n\cdot\nabla \uu_n, \mathbf{A}\uu_n)|
&\leq \|\uu_n\|_{\mathbf{L}^6}\|\nabla\uu_n\|_{\mathbf{L}^3}\|\mathbf{A}\uu_n\|\non\\
& \leq  C\|\nabla \uu_n\|^\frac32\|\mathbf{A}\uu_n\|^\frac{3}{2} \notag \\
& \leq \dfrac{\nu_*}{12}\|\mathbf{A}\uu_n\|^2 + C\|\nabla \uu_n\|^6,
\label{term 1}
 \end{align}
and
\begin{align}
| (\mu_n\nabla\phi_n, \mathbf{A}\uu_n)| &=| ((\mu_n-\overline{\mu_n})\nabla\phi_n, \mathbf{A}\uu_n)| \notag\\
& \leq \|\nabla \phi_n\|_{\mathbf{L}^6} \|\mu_n-\overline{\mu_n}\|_{L^3}\|\mathbf{A}\uu_n\|
\notag \\
& \leq \dfrac{\nu_*}{12}\|\mathbf{A}\uu_n\|^2 +  C\|\nabla \mu_n\|^2,
\notag
\end{align}
where we have used the fact $\mathbf{A}\uu_n\in \mathbf{H}_\sigma$ so that $(\overline{\mu_n}\nabla\phi_n, \mathbf{A}\uu_n)=0$ after intrgration by parts. The term on the left-hand side of \eqref{duu} involving the nonconstant viscosity can be treated as in \cite[Section 5]{GMT18}. Here the proof is indeed easier thanks to the estimate \eqref{unilow} and the Sobolev embedding theorem $H^2(\Omega)\hookrightarrow L^\infty(\Omega)$ in three dimensions. Let us sketch the procedure. It follows from Lemma \ref{stokes} that there exists a function $P_n \in L^2(0,T;H^1(\Omega))$ such that $-\Delta \uu_n + \nabla P_n = \mathbf{f}:=\mathbf{A}\uu_n$ almost everywhere in $\Omega \times (0,T)$. Then we find that
		\begin{align}
			&	- 2(\nabla \cdot(\nu(\phi_n)D\uu_n), \mathbf{A}\uu_n) \notag\\
& \quad = -(\nu(\phi_n)\Delta\uu_n,\mathbf{A}\uu_n) - 2(D\uu_n \nabla\nu(\phi_n), \mathbf{A}\uu_n)\notag \\
			& \quad = (\nu(\phi_n)\mathbf{A}\uu_n, \mathbf{A}\uu_n)
- (\nu(\phi_n)\nabla P_n, \mathbf{A}\uu_n) - 2 \left(\nu'(\phi_n)D\uu_n \nabla \phi_n, \mathbf{A}\uu_n \right) \notag \\
			& \quad \geq \nu_*\|\mathbf{A}\uu_n\|^2 + \left(\nu'(\phi_n)P_n  \nabla \phi_n, \mathbf{A}\uu_n \right) - 2 \left(\nu'(\phi_n)D\uu_n \nabla \phi_n, \mathbf{A}\uu_n \right).
\label{eq:step42}
		\end{align}
The last two terms on the right-hand side of \eqref{eq:step42} can be estimated as follows
\begin{align}
 		& |\left(\nu'(\phi_n)P_n  \nabla \phi_n, \mathbf{A}\uu_n \right)| + |2 \left(\nu'(\phi_n)D\uu_n \nabla \phi_n, \mathbf{A}\uu_n \right)|\notag  \\
		& \quad  \leq C \|\nu'(\phi_n)\|_{L^\infty}\|\nabla \phi_n\|_{\mathbf{L}^6} \left( \|P_n\|_{L^3} + \|D\uu_n\|_{\mathbf{L}^3} \right) \|\mathbf{A}\uu_n\|\notag \\
		& \quad  \leq C\left(\|P_n\|^\frac{1}{2}\|P_n\|_{H^1(\Omega)}^\frac{1}{2} + \|\nabla \uu_n\|^\frac{1}{2}\|\mathbf{A}\uu_n\|^\frac{1}{2} \right)\|\mathbf{A}\uu_n\|
\notag \\	
		& \quad  \leq C\left( \|\nabla \uu_n\|^\frac{1}{4}\|\mathbf{A}\uu_n\|^\frac{3}{4} + \|\nabla \uu_n\|^\frac{1}{2}\|\mathbf{A}\uu_n\|^\frac{1}{2} \right)\|\mathbf{A}\uu_n\| \notag \\
		& \quad  \leq \frac{\nu_*}{12}\|\mathbf{A}\uu_n\|^2  + C\|\nabla \uu_n\|^2.
\notag
\end{align}
Then we infer from the above estimates and \eqref{duu} that
\be
\frac12 \frac{\mathrm{d}}{\mathrm{d}t} \|\nabla \uu_n\|^2+ \frac{3\nu_*}{4} \|\mathbf{A}\uu_n\|^2
\leq C\|\nabla \uu_n\|^6 + C(\|\nabla \uu_n\|^2 + \|\nabla \mu_n\|^2).
\label{duu-A}
\ee
where the constant $C>0$ depends on $\|\uu_0\|$, $\|\phi_0\|_{H^2}$, $\Omega$, $\eta$, but is independent of $n$ and $t$.

Next, testing the equation for $\uu_n$ by $\partial_t\uu_n$, we get
\begin{align}
\|\partial_t\uu_n\|^2=-((\uu_n\cdot\nabla)\uu_n,\partial_t\uu_n) +(\nabla\cdot(2\nu(\phi_n)D\uu_n), \partial_t\uu_n)
+ (\mu_n\nabla \phi_n, \partial_t \uu_n).
\label{duu-ta}
\end{align}
The three terms on the right-hand side of \eqref{duu-ta} can be estimated as in \cite{GMT18}, with minor modifications due to the estimate \eqref{unilow}:
\begin{align}
|((\uu_n\cdot\nabla)\uu_n,\partial_t\uu_n)|&\leq \|\uu_n\|_{\mathbf{L}^6}\|\nabla \uu_n\|_{\mathbf{L}^3}\|\partial_t\uu_n\|\notag\\
&\leq C\|\nabla \uu_n\|^\frac32\|\mathbf{A}\uu_n\|^\frac12\|\partial_t\uu_n\|\notag\\
&\leq \frac{1}{12}\|\partial_t\uu_n\|^2+\|\mathbf{A}\uu_n\|^2 +C\|\nabla \uu_n\|^6,
\label{duu-ta1}
\end{align}
\begin{align*}
&|(\nabla\cdot(2\nu(\phi_n)D\uu_n), \partial_t\uu_n)|\\
&\quad \leq |(\nu(\phi_n)\Delta\uu_n, \partial_t\uu_n)|+ 2|(\nu'(\phi_n)D\uu_n\nabla \phi_n,\partial_t\uu_n)|\\
&\quad \leq C\|\nu(\phi_n)\|_{L^\infty}\|\mathbf{A}\uu_n\|\|\partial_t\uu_n\|+ C\|\nu'(\phi_n)\|_{L^\infty}\|D\uu_n\|_{\mathbf{L}^3}\|\nabla \phi_n\|_{\mathbf{L}^6}\|\partial_t\uu_n\|\\
&\quad \leq C\|\mathbf{A}\uu_n\|\|\partial_t\uu_n\| +C\|\nabla \uu_n\|^\frac12\|\mathbf{A}\uu_n\|^\frac12\|\partial_t\uu_n\|\\
&\quad \leq \frac{1}{12}\|\partial_t\uu_n\|^2 +C\|\mathbf{A}\uu_n\|^2 +C\|\nabla \uu_n\|^2,
\end{align*}
and
\begin{align}
|(\mu_n\nabla \phi_n, \partial_t \uu_n)|&=|((\mu_n-\overline{\mu_n})\nabla \phi_n, \partial_t \uu_n)|\notag\\
&\leq \|\mu_n-\overline{\mu_n}\|_{L^3}\|\nabla \phi_n\|_{\mathbf{L}^6}\|\partial_t\uu_n\|\notag \\
&\leq C\|\nabla \mu_n\|\|\partial_t\uu_n\| \notag \\
&\leq  \frac{1}{12}\|\partial_t\uu_n\|^2 +C\|\nabla \mu_n\|^2.
\label{duu-ta3}
\end{align}
Thus, we can deduce from the above estimates and \eqref{duu-ta} that
\begin{align}
\frac{3}{4}\|\partial_t\uu_n\|^2\leq C_*\|\mathbf{A}\uu_n\|^2 + C\|\nabla \uu_n\|^6 + C\|\nabla \uu_n\|^2 +C\|\nabla \mu_n\|^2,
\label{duu-ta2}
\end{align}
where $C_*, C>0$ are some constants depending on $\|\uu_0\|$, $\|\phi_0\|_{H^2}$, $\eta$ and $\Omega$, but independent of $n$ and $t$.
Multiplying the inequality \eqref{duu-ta2} by $\frac{\nu_*}{4C_*}>0$ and adding the resultant with \eqref{duu-A}, we arrive at the conclusion \eqref{duu-A1}.

The proof is complete.
\end{proof}
\begin{lemma}\label{lem:dmu}
The approximate solution $(\uu_n,\phi_n,\mu_n)$ satisfies
\begin{align}
& \frac{\mathrm{d}}{\mathrm{d}t}\left(\frac12\int_\Omega m(\phi_n)|\nabla \mu_n|^2\,\mathrm{d}x+ \big(\uu_n\cdot\nabla \phi_n,\mu_n\big)\right)
+\frac34\|\Delta\partial_t\phi_n\|^2\notag\\
&\quad \leq \frac{1}{16} \|\partial_t\uu_n\|^2 + C\|\nabla \mu_n\|^4+ C\|\nabla \mu_n\|^2+ C\|\nabla \uu_n\|^2,
\label{dmu-1}
\end{align}
where the constant $C>0$ depends on $\|\uu_0\|$, $\|\phi_0\|_{H^2}$, $\Omega$, $\eta$, but is independent of $n$ and $t$.
\end{lemma}
\begin{proof}
Testing the second equation in \eqref{eq:weak} by $\partial_t \mu_n$, we obtain
\begin{align}
(\partial_t\phi_n, \partial_t \mu_n)+(\uu_n\cdot\nabla\phi_n, \partial_t \mu_n)=-(m(\phi_n)\nabla \mu_n, \nabla \partial_t \mu_n).
\label{dmu-3}
\end{align}
On the other hand, differentiating the fourth equation in \eqref{eq:weak} with respect to time and testing the resultant by $\partial_t\phi_n\in W_n$, we get
\begin{align}
(\partial_t \mu_n, \partial_t\phi_n)&=(\partial_t \Pi_n \Delta^2\phi_n, \partial_t \phi_n) + (\partial_t \Pi_n L(\phi_n),\partial_t \phi_n)\notag\\
&=(\Delta^2\partial_t \phi_n, \partial_t \phi_n) + (\partial_t L(\phi_n),\partial_t \phi_n),
\label{dmu-4}
\end{align}
where $L(\phi_n)$ is given by \eqref{LLL} and $\omega_n=-\Delta\phi_n+f(\phi_n)$.
Thus, due to \eqref{dmu-4} and integration by parts, \eqref{dmu-3} can be written as
\begin{align}
&\frac{\mathrm{d}}{\mathrm{d}t}\left(\frac12\int_\Omega m(\phi_n)|\nabla \mu_n|^2\,\mathrm{d}x+ \big(\uu_n\cdot\nabla \phi_n,\mu_n\big)\right)
+ \|\Delta \partial_t \phi_n\|^2 \notag\\
&\quad = \frac12 \int_\Omega m'(\phi_n)\partial_t\phi_n|\nabla\mu_n|^2\,\mathrm{d}x + (\partial_t\uu_n\cdot\nabla \phi_n, \mu_n) + (\uu_n\cdot\nabla \partial_t\phi_n,\mu_n)\notag\\
&\qquad -(\partial_t L(\phi_n),\partial_t \phi_n)\notag\\
&\quad =: I_1+I_2+I_3+I_4.
\label{dmu-2}
\end{align}
Since $\overline{\partial_t\phi_n}=0$, we infer from the Poincar\'{e}--Wirtinger inequality that
$$
\|\partial_t\phi_n\|_{V_0}\leq C\|\nabla \partial_t\phi_n\|,\quad \|\partial_t\phi_n\|_{H^2}\leq C\|\Delta \partial_t\phi_n\|,
$$
where $C>0$ only depends on $\Omega$. Besides, from the Gelfand triple $V_0^*\hookrightarrow L^2(\Omega)\cong L^2(\Omega)\hookrightarrow V_0$, we have the interpolation inequality
$\|\partial_t\phi_n\|\leq C\|\partial_t\phi_n\|_{V_0^*}^\frac12\|\nabla \partial_t\phi_n\|^\frac12$.
This combined with the fact $\|\nabla \partial_t\phi_n\|^2\leq \|\partial_t\phi_n\|\|\Delta\partial_t\phi_n\|$ yields that
$$
\|\partial_t\phi_n\|\leq C\|\partial_t\phi_n\|_{V_0^*}^\frac23\|\Delta \partial_t\phi_n\|^\frac13.
$$
As a consequence, the term $I_1$ can be estimated by using Agmon's inequality such that
\begin{align}
I_1&\leq \frac12 \|m'(\phi_n)\|_{L^\infty}\|\partial_t\phi_n\|_{L^\infty}\|\nabla \mu_n\|^2\notag \\
&\leq C \|\partial_t\phi_n\|_{H^1}^\frac12\|\partial_t\phi_n\|_{H^2}^\frac12\|\nabla \mu_n\|^2 \notag \\
&\leq C \|\partial_t\phi_n\|^\frac14\|\Delta\partial_t\phi_n\|^\frac34\|\nabla \mu_n\|^2 \notag \\
&\leq C \|\partial_t\phi_n\|_{V_0^*}^\frac16\|\Delta\partial_t\phi_n\|^\frac56\|\nabla \mu_n\|^2 \notag \\
&\leq \frac{1}{12} \|\Delta\partial_t\phi_n\|^2+ C\|\partial_t\phi_n\|_{V_0^*}^2+ C\|\nabla \mu_n\|^4.
\notag
\end{align}
Since
\begin{align*}
\|\partial_t\phi_n\|_{V_0^*} & =\sup_{w\in V_0,\ \|w_0\|_{V_0}=1} |(\uu_n\cdot\nabla \phi_n,w)+(m(\phi_n)\nabla \mu_n, \nabla w)|\\
&\leq C(\|\nabla \uu_n\|+\|\nabla \mu_n\|),
\end{align*}
we have
$$
I_1\leq \frac{1}{12} \|\Delta\partial_t\phi_n\|^2+ C\|\nabla \mu_n\|^4 +C\|\nabla \mu_n\|^2 + C\|\nabla \uu_n\|^2.
$$
The term $I_2$ can be estimated in the same way as in \eqref{duu-ta3} by choosing a different coefficient in Young's inequality such that
$$
I_2\leq \frac{1}{16}\|\partial_t\uu_n\|^2+C\|\nabla \mu_n\|^2.
$$
Next for $I_3$, we have
\begin{align*}
I_3&=(\uu_n\cdot\nabla \partial_t\phi_n,\mu_n-\overline{\mu_n})\\
&\leq \|\uu_n\|\|\nabla \partial_t\phi_n\|_{\mathbf{L}^6}\|\mu_n-\overline{\mu_n}\|_{L^3}\\
&\leq C\|\Delta\partial_t \phi\|\|\nabla \mu_n\|\\
&\leq \frac{1}{12} \|\Delta\partial_t\phi_n\|^2 +C\|\nabla \mu_n\|^2.
\end{align*}
A direct calculation yields that
\begin{align*}
\partial_t L(\phi_n)
&=\big[-2f''(\phi_n)\Delta\phi_n +f''(\phi_n)f(\phi_n) -f'''(\phi_n)|\nabla \phi_n|^2 +(f'(\phi_n)+\eta)f'(\phi_n)\big]\partial_t\phi_n\\
&\quad -2f''(\phi_n)\nabla \phi_n\cdot\nabla \partial_t\phi_n-(2f'(\phi_n)+\eta)\Delta\partial_t\phi_n.
\end{align*}
Thus, the term $I_4$ can be estimated by
\begin{align*}
I_4&\leq |(\partial_t L(\phi_n),\partial_t\phi_n)|\\
&\leq 2\|f''(\phi_n)\|_{L^\infty}\|\Delta\phi_n\|\|\partial_t\phi_n\|_{L^4}^2 + \|f''(\phi_n)\|_{L^\infty}\|f(\phi_n)\|_{L^\infty}\|\partial_t\phi_n\|^2\\
&\quad +\|f'''(\phi_n)\|_{L^\infty}\|\nabla \phi_n\|_{\mathbf{L}^6}^2\|\partial_t\phi_n\|_{L^3}^2
+\|f'(\phi_n)+\eta\|_{L^\infty}\|f'(\phi_n)\|_{L^\infty}\|\partial_t\phi_n\|^2\\
&\quad +2 \|f''(\phi_n)\|_{L^\infty}\|\nabla \phi_n\|_{\mathbf{L}^6}\|\nabla \partial_t\phi_n\|\|\partial_t\phi_n\|_{\mathbf{L}^3}
+\|2f'(\phi_n)+\eta\|_{L^\infty}\|\Delta\partial_t\phi_n\|\|\partial_t\phi_n\|\\
&\leq C\|\nabla \partial_t\phi_n\|^2+C\|\Delta\partial_t\phi_n\|\|\partial_t\phi_n\|\\
&\leq C\|\Delta\partial_t\phi_n\|\|\partial_t\phi_n\|\\
&\leq C\|\Delta\partial_t\phi_n\|^\frac43\|\partial_t\phi_n\|_{V_0^*}^\frac23\\
&\leq \frac{1}{12} \|\Delta\partial_t\phi_n\|^2 +C\|\nabla \mu_n\|^2 + C\|\nabla \uu_n\|^2.
\end{align*}
Collecting the above estimates, we can deduce the required inequality \eqref{dmu-1} from \eqref{dmu-2}.

The proof is complete.
\end{proof}

\textbf{Proof of Theorem \ref{thm:locstr}}. We are in a position to complete the proof of Theorem \ref{thm:locstr}.
Noticing that
\begin{align*}
|(\uu_n\cdot\nabla \phi_n,\mu_n)| &=|(\uu_n\cdot\nabla \phi_n,\mu_n-\overline{\mu_n})|\\
&\leq \|\uu_n\|_{\mathbf{L}^3}\|\nabla \phi_n\|_{\mathbf{L}^3}\|\mu_n-\overline{\mu_n}\|_{L^3}\\
&\leq C\|\nabla \phi_n\|_{\mathbf{L}^3}\|\nabla \uu_n\|\|\nabla \mu_n\|\\
&\leq \frac{C_1}{4m_*}\|\nabla \uu_n\|^2+\frac{m_*}{4}\|\nabla \mu_n\|^2,
\end{align*}
where $m_*>0$ is given by the assumption (A2) and $C_1>0$ depends on $\|\uu_0\|$, $\|\phi_0\|_{H^2}$, $\Omega$, $\eta$, but is independent of $n$ and $t$.
Set
\begin{equation}
\gamma=\min\left\{\frac{m_*}{C_1},\ \frac{\nu_*}{C_*}\right\}>0,\label{gamma}
\end{equation}
where $C_*>0$ is the constant specified in Lemma \ref{lem:duu}.
Multiplying \eqref{dmu-1} by $\gamma$ and adding the resultant with \eqref{duu-A1}, we find that
\begin{align}
&\frac{\mathrm{d}}{\mathrm{d}t} \left(\frac12 \|\nabla \uu_n\|^2+\frac{\gamma}{2}\int_\Omega m(\phi_n)|\nabla \mu_n|^2\,\mathrm{d}x +\gamma \big(\uu_n\cdot\nabla \phi_n,\mu_n\big)\right)\notag \\
&\qquad +\frac{\nu_*}{2}\|\mathbf{A}\uu_n\|^2 + \frac{\nu_*}{8C_*}\|\partial_t\uu_n\|^2 +\frac{3\gamma}{4} \|\Delta\partial_t\phi_n\|^2\notag\\
&\quad \leq C\|\nabla \uu\|^6 + C\gamma \|\nabla \mu_n\|^4 + C(1+\gamma)(\|\nabla \uu_n\|^2 +\|\nabla \mu_n\|^2).
\label{duu-mu}
\end{align}

Set
$$
\Lambda_n(t) = \frac12 \|\nabla \uu_n(t)\|^2+\frac{\gamma}{2}\int_\Omega m(\phi_n(t))|\nabla \mu_n(t)|^2\,\mathrm{d}x
+ \gamma (\uu_n(t)\cdot\nabla \phi_n(t),\mu_n(t)).
$$
Thanks to the assumption (A2) and the uniform estimate \eqref{unilow}, there exists some positive constant $m^*$ that is independent of $n$ satisfying $m^*>m_*$ and $\|m(\phi_n)\|_{L^\infty}\leq m^*$. As a consequence, it holds
\begin{align}
&\Lambda_n(t)\leq \|\nabla \uu_n(t)\|^2+ \gamma m^*\|\nabla \mu_n(t)\|^2,\label{LL1}\\
&\Lambda_n(t)\geq \frac14\|\nabla \uu_n(t)\|^2+ \frac{\gamma m_*}{4}\|\nabla \mu_n(t)\|^2.\label{LL2}
\end{align}
Besides, it follows from \eqref{duu-mu} that
\begin{align}
\frac{\mathrm{d}}{\mathrm{d}t}\Lambda_n(t)  +\frac{\nu_*}{2}\|\mathbf{A}\uu_n\|^2 + \frac{\nu_*}{8C_*}\|\partial_t\uu_n\|^2 +\frac{3\gamma}{4} \|\Delta\partial_t\phi_n\|^2\leq C\Lambda_n(t)^3,
\label{La}
\end{align}
where the positive constant $C$ depends on $\|\uu_0\|$, $\|\phi_0\|_{H^2}$, $\Omega$, $\eta$, $\nu_*$, $m_*$,  but is independent of $n$ and $t$.

From the assumption on the initial data, we easily  find that $\Lambda_n(0)\leq C$, where $C$ depends on $\|\uu_0\|_{\mathbf{V}_\sigma}$, $\|\phi_0\|_{H^5}$, $\eta$, $\nu_*$, $m_*$,  $\Omega$, but again is independent of $n$. Then we can deduce from the differential inequality \eqref{La} that there exists a finite time $T_0>0$, depending on $\|\uu_0\|_{\mathbf{V}_\sigma}$, $\|\phi_0\|_{H^5}$, $\eta$, $\nu_*$, $m_*$ and $\Omega$ such that
$$
\Lambda_n(t)+\int_0^{T_0} \left(\|\mathbf{A}\uu_n(t)\|^2 + \|\partial_t\uu_n(t)\|^2 + \|\Delta\partial_t\phi_n(t)\|^2\right)\mathrm{d}t\leq C,
$$
for all $t\in [0,T_0]$.

Thus, the above estimate together with the up/lower bounds for $\Lambda(t)$ implies that
\begin{align*}
&\uu_n\ \text{is bounded in}\ L^\infty(0,T_0;\mathbf{V}_\sigma)\cap L^2(0,T_0;\mathbf{W}_\sigma)\cap H^1(0,T_0;\mathbf{H}_\sigma),\\
&\phi_n\ \text{is bounded in}\ H^1(0,T_0;H^2(\Omega)),\\
&\mu_n\ \text{is bounded in}\ L^\infty(0,T_0; H^1(\Omega)),
\end{align*}
all uniformly in $n$. Combining these with the known uniform estimates obtained in Section \ref{sec:weak}, by a standard compactness argument, we can pass to the limit as $n\to+\infty$ and find a convergent subsequence, whose limit denoted by $(\uu, \phi, \mu,\omega)$ enjoys the same regularity indicated above and satisfies the original system a.e. in $\Omega\times (0,T_0)$. Moreover, the boundary and initial conditions are satisfied such that  $\uu=\mathbf{0}$ and $\partial_\mathbf{n}\phi=\partial_\mathbf{n}\Delta\phi=0$ on $\partial\Omega$; $\uu(\cdot,0)=\uu_0$ and $\phi(\cdot,0)=\phi_0$ in $\Omega$. To improve the spatial regularity, we can apply the elliptic estimate like in the previous section such that
$$
\omega\in L^\infty(0,T_0;H^3(\Omega)),\quad \phi\in L^\infty(0,T_0;H^5(\Omega)).
$$
To improve the regularity of $\mu$, we consider the Neumann problem with variable coefficient
\begin{equation}
\begin{cases}
-\nabla \cdot(m(\phi)\nabla u)=f,\quad \text{in}\ \Omega,\\
\partial_\mathbf{n}u=0,\qquad\qquad \qquad \ \text{on}\ \partial\Omega.
\end{cases}\label{Neu}
\end{equation}
We note that $\phi\in L^\infty(0,T_0; H^5(\Omega))$ and $0<m_*\leq \|m(\phi)\|_{L^\infty}\leq m^*$ for some $m^*>m_*>0$. Besides, from
\begin{align*}
\|\uu\cdot\nabla \phi\|_{H^2}&\leq C(\|\uu\|_{\mathbf{L}^\infty}\|\nabla \phi\|_{\mathbf{H}^2}+\|\mathbf{A}\uu\|\|\nabla \phi\|_{\mathbf{L}^\infty})\\
&\leq C\|\mathbf{A}\uu\|^\frac12 + C\|\mathbf{A}\uu\|,
\end{align*}
we have $\uu\cdot\nabla \phi\in L^2(0,T_0;H^2(\Omega))$. Using the above facts and $\partial_t\phi\in L^2(0,T_0;H^2(\Omega))$, we can apply \cite[Theorem 2.1]{Gi20} to problem \eqref{Neu} with $u=\mu$ and $f=\partial_t\phi +\uu\cdot\nabla \phi$ to conclude that
$$
\|\mu\|_{H^2}\leq C(1+\|\phi\|_{W^{1,\infty}})\|\partial_t\phi +\uu\cdot\nabla \phi\|,
$$
and thus $\mu\in L^2(0,T_0;H^2_N(\Omega))$. From the following elliptic estimate for $\omega$
\begin{align*}
\|\omega\|_{H^4}& \leq C(\| \mu-f'(\phi)\omega-\eta\omega\|_{H^2}+\|\omega\|)\notag\\
& \leq C(\|\mu\|_{H^2}+\|f'(\phi)+\eta\|_{L^\infty}\|\omega\|_{H^2} + \|f'(\phi)+\eta\|_{H^2}\|\omega\|_{L^\infty}+\|\omega\|),
\end{align*}
we find that $\omega\in L^2(0,T_0,H^4(\Omega))$. This combined with the elliptic estimate for $\phi$
\begin{align*}
\|\phi\|_{H^6}
&\leq C(\|\omega-f(\phi)\|_{H^4}+\|\phi\|)\\
&\leq C\|\omega\|_{H^4} +C(\|\phi\|_{L^\infty}^2\|\phi\|_{H^4}+\|\phi\|_{H^4}),
\end{align*}
further leads to $\phi\in L^2(0,T_0,H^6(\Omega))$. Finally, we note that the pressure
$P\in L^\infty(0,T_0;H^1(\Omega))$ can be recovered through the De Rham theorem (see e.g., \cite{Te}).

Concerning the uniqueness, we note that every strong solution to problem \eqref{NS}--\eqref{IC} is also a weak solution. Moreover, the regularity of strong solution
$\uu\in L^\infty(0,T_0;\mathbf{V}_\sigma)\cap L^2(0,T_0;\mathbf{W}_\sigma)$ implies that $\uu\in L^4(0,T_0;\mathbf{L}^6(\Omega))\cap L^\frac{4}{3}(0,T_0;\mathbf{W}^{1,6}(\Omega))$, which verifies the condition $\mathrm{(a)}$ with $p=6$, $q=4$ and condition $\mathrm{(b)}$ with $p=6$, $q=\frac{4}{3}$ in Theorem \ref{thm:weakuni}. Therefore, the strong solution obtained above is indeed unique.

The proof of Theorem \ref{thm:locstr} is complete. \hfill $\square$

\begin{remark}
For strong solutions, a continuous dependence estimate for $(\uu, \phi)$ with respect the initial data in $\mathbf{L}^2(\Omega)\times H^2(\Omega)$ can be derived by using the standard energy method. On the other hand, since Theorem \ref{thm:weakuni} provides a comparison between two weak solutions of which only one possesses extra regularity, it combined with Theorems \ref{thm:weak} and \ref{thm:locstr} indeed implies a weak-strong uniqueness result for problem \eqref{NS}--\eqref{IC}. Namely, whenever a strong solution exists, a weak solution emanating from the same initial data coincides with the strong solution on its existence interval.
\end{remark}

\begin{remark}
We note that the boundary condition \eqref{boundary} does not imply  $\partial_\mathbf{n}\Delta^2 \phi=0$ on $\partial \Omega$ for the solution to  problem \eqref{NS}--\eqref{IC} whenever the trace of the normal derivative of $\Delta^2\phi$  makes sense. Thus, we cannot apply the argument in the periodic setting (see \cite{CW20,WX13}) to derive a global-in-time estimate on $\|\nabla \Delta \phi\|$ by testing the phase-field equation \eqref{phasefield} with $-\Delta^3 \phi$ (or equivalently, testing the equation \eqref{pot1} by $\partial_t\Delta \phi$). Actually, we have $\partial_\mathbf{n}\Delta^2\phi=\partial_\mathbf{n}\Delta f(\phi)= 6\phi\nabla(|\nabla \phi|^2)\cdot\mathbf{n}$, which does not vanish  on $\partial\Omega$ (cf. \cite{CG18,CG19}).
\end{remark}

\subsection{Blow-up criteria}
In this section, we prove Theorem \ref{thm:buc}. It turns out that the blow-up criteria for problem \eqref{NS}--\eqref{IC} again only involve the velocity field and is close to those for the Navier--Stokes equations in a bounded three dimensional domain (cf. e.g., \cite{Be02, Ga00,Se62}).\smallskip

\textbf{Proof of  Theorem \ref{thm:buc}}.
Since we have already shown that problem  \eqref{NS}--\eqref{IC} admits a unique local strong solution for any $(\uu_0, \phi_0) \in \mathbf{V}_\sigma\times (H^5(\Omega)\cap H^4_N(\Omega))$, we only need to establish a priori estimates for the strong solution. Most estimates can be carried out along the lines for approximate solutions in the previous subsection, except two terms in \eqref{term 1} and \eqref{duu-ta1}.

We first consider condition $\mathrm{(a)}$. For any $p>3$, we deduce that (cf. \cite{Ga00})
\begin{align*}
|((\uu\cdot\nabla)\uu,\mathbf{A}\uu)|
&\leq \|\uu\|_{\mathbf{L}^p}\|\nabla \uu\|_{\mathbf{L}^\frac{2p}{p-2}}\|\mathbf{A}\uu\|\notag\\
&\leq
C\|\uu\|_{\mathbf{L}^{p}} \left(\|\mathbf{A} \uu\|^\frac{3}{p}\|\nabla \uu\|^\frac{p-3}{p}\right)\|\mathbf{A}\uu\|
\notag\\
&\leq \frac{\nu_*}{12}\|\mathbf{A}\uu\|^2 +C\|\uu\|^\frac{2p}{p-3}_{\mathbf{L}^p}\|\nabla
\uu\|^2,
\end{align*}
and
\begin{align*}
|((\uu\cdot\nabla)\uu,\partial_t\uu)|&\leq \|\uu\|_{\mathbf{L}^p}\|\nabla \uu\|_{\mathbf{L}^\frac{2p}{p-2}}\|\partial_t\uu\|\notag\\
&\leq C\|\uu\|_{\mathbf{L}^{p}}\left(\|\mathbf{A} \uu\|^\frac{3}{p}\|\nabla \uu\|^\frac{p-3}{p}\right)\|\partial_t\uu\|  \notag\\
&\leq \frac{1}{12}\|\partial_t\uu\|^2+\|\mathbf{A}\uu\|^2 +C\|\uu\|^\frac{2p}{p-3}_{\mathbf{L}^p}\|\nabla
\uu\|^2.
\end{align*}
Introduce
\begin{align}
\Lambda(t) = \frac12 \|\nabla \uu(t)\|^2+\frac{\gamma}{2}\int_\Omega m(\phi(t))|\nabla \mu(t)|^2\,\mathrm{d}x+ \gamma (\uu(t)\cdot\nabla \phi(t),\mu(t)),
\label{LL3}
\end{align}
where the constant $\gamma>0$ is chosen as in \eqref{gamma}, depending on $\|\uu_0\|$, $\|\phi_0\|_{H^2}$, $\Omega$, $\eta$, $\nu_*$, $m_*$, but independent of $t$.
The quantity $\Lambda(t)$ also fulfills similar inequalities like \eqref{LL1}, \eqref{LL2}. In analogy to \eqref{duu-mu}, we can deduce that
\begin{align}
&\frac{\mathrm{d}}{\mathrm{d}t} \Lambda(t)  +\frac{\nu_*}{2}\|\mathbf{A}\uu\|^2 + \frac{\nu_*}{8C_*}\|\partial_t\uu\|^2 +\frac{3\gamma}{4}\|\Delta\partial_t\phi\|^2\notag\\
&\quad \leq C\|\uu\|^\frac{2p}{p-3}_{\mathbf{L}^p}\|\nabla
\uu\|^2 + C\gamma \|\nabla \mu\|^4 + C(1+\gamma)(\|\nabla \uu\|^2 + \|\nabla \mu\|^2)\notag\\
&\quad \leq C(1+\gamma)\left( \|\uu\|^\frac{2p}{p-3}_{\mathbf{L}^p}+ \|\nabla \mu\|^2+1\right)\Lambda(t).
\label{duu-mur1}
\end{align}
It follows from \eqref{int} and condition $\mathrm{(a)}$ that $ \|\uu\|^\frac{2p}{p-3}_{\mathbf{L}^p}+\|\nabla\mu\|^2\in L^1(0,T_1)$. Hence, from Gronwall's lemma, we obtain
 $$
 \Lambda(t)\leq C_{T_1}<+\infty,\quad \forall\,t\in [0,T_1],
 $$
which implies that the $\mathbf{V}_\sigma\times H^1$ norm of  $(\uu, \mu)$ is bounded on $[0,T_1]$. From the elliptic estimate, $\|\phi(t)\|_{H^5}$ is also bounded on $[0,T_1]$. Thus, $[0, T_1)$ cannot be the maximal interval of existence, and the local strong solution $(\uu, \phi)$ can be extended beyond $T_1$.

Next, we consider condition  $\mathrm{(b)}$. For any  $p\in (\frac{3}{2},3)$, from the Sobolev embedding theorem in three dimensions, we take $r=\frac{3p}{3-p}\in (3,+\infty)$ such that such that $\mathbf{W}^{1,p}(\Omega)\hookrightarrow \mathbf{L}^r(\Omega)$. Then it holds
\begin{align*}
|((\uu\cdot\nabla)\uu,\mathbf{A}\uu)|
&\leq C\|\uu\|_{\mathbf{L}^r}\|\nabla \uu\|_{\mathbf{L}^\frac{2r}{r-2}}\|\mathbf{A}\uu\|\notag\\
&\leq
C\|\nabla \uu\|_{\mathbf{L}^{p}}\|\nabla \uu\|_{\mathbf{L}^\frac{6p}{5p-6}}\|\mathbf{A}\uu\|
\notag\\
&\leq C\|\nabla \uu\|_{\mathbf{L}^{p}}\|\nabla\uu\|^{2-\frac{3}{p}}\|\mathbf{A}\uu\|^\frac{3}{p}\\
&\leq \frac{\nu_*}{12}\|\mathbf{A}\uu\|^2 +C\|\nabla \uu\|^\frac{2p}{2p-3}_{\mathbf{L}^p}\|\nabla
\uu\|^2,
\end{align*}
and in a similar manner,
\begin{align*}
|((\uu\cdot\nabla)\uu,\partial_t\uu)|
&\leq C\|\uu\|_{\mathbf{L}^r}\|\nabla \uu\|_{\mathbf{L}^\frac{2r}{r-2}}\|\partial_t\uu\|\notag\\
&\leq C\|\nabla \uu\|_{\mathbf{L}^{p}}\|\nabla\uu\|^{2-\frac{3}{p}}\|\mathbf{A}\uu\|^{\frac{3}{p}-1}\|\partial_t\uu\| \notag\\
&\leq \frac{1}{12}\|\partial_t\uu\|^2+\|\mathbf{A}\uu\|^2 +C\|\nabla \uu\|^\frac{2p}{2p-3}_{\mathbf{L}^p}\|\nabla
\uu\|^2.
\end{align*}
If $p= 3$, we simply have
\begin{align*}
|((\uu\cdot\nabla)\uu,\mathbf{A}\uu)|
&\leq C\|\uu\|_{\mathbf{L}^6}\|\nabla \uu\|_{\mathbf{L}^3}\|\mathbf{A}\uu\|\notag\\
&\leq
C\|\nabla \uu\|\|\nabla \uu\|_{\mathbf{L}^3}\|\mathbf{A}\uu\|
\notag\\
&\leq \frac{\nu_*}{12}\|\mathbf{A}\uu\|^2 +C\|\nabla \uu\|^2_{\mathbf{L}^3}\|\nabla
\uu\|^2,
\end{align*}
and
\begin{align*}
|((\uu\cdot\nabla)\uu,\partial_t\uu)|
&\leq C\|\uu\|_{\mathbf{L}^6}\|\nabla \uu\|_{\mathbf{L}^3}\|\partial_t\uu\|\notag\\
&\leq C\|\nabla \uu\|\|\nabla \uu\|_{\mathbf{L}^3}\|\partial_t\uu\| \notag\\
&\leq \frac{1}{12}\|\partial_t\uu\|^2 +C\|\nabla \uu\|^2_{\mathbf{L}^3}\|\nabla
\uu\|^2.
\end{align*}
Hence, we can deduce that for $p\in (\frac{3}{2},3]$
\begin{align}
&\frac{\mathrm{d}}{\mathrm{d}t} \Lambda(t)  +\frac{\nu_*}{2}\|\mathbf{A}\uu\|^2 + \frac{\nu_*}{8C_*}\|\partial_t\uu\|^2 +\frac{3\gamma}{4}\|\Delta\partial_t\phi\|^2\notag\\
&\quad \leq C\|\nabla \uu\|^\frac{2p}{2p-3}_{\mathbf{L}^p}\|\nabla
\uu\|^2 + C\gamma \|\nabla \mu\|^4 + C(1+\gamma)(\|\nabla \uu\|^2 + \|\nabla \mu\|^2)\notag\\
&\quad \leq C(1+\gamma)\left( \|\nabla \uu\|^\frac{2p}{2p-3}_{\mathbf{L}^p}+ \|\nabla \mu\|^2+1\right)\Lambda(t).
\label{duu-mur2}
\end{align}
Thanks to \eqref{int} and condition $\mathrm{(b)}$, we arrive at the conclusion by applying Gronwall's lemma.

The proof of Theorem \ref{thm:buc} is complete.
\hfill $\square$

\section{Appendix: Model Derivation}\label{sec:derivation}
\setcounter{equation}{0}
The energetic variational approach gives a unified variational framework in studying complex fluids with micro-structures (see e.g., \cite{BKL16, DLRW09, HKL, GKL18}). In this appendix, we present a formal derivation of the hydrodynamic system \eqref{NS}--\eqref{pot2} via the energetic variational approach, which provides a further understanding of the complex membrane-fluid interactions. In order to derive a system for the conserved order parameter $\phi$ (corresponding to the volume conservation of the vesicle), we adapt the argument presented in \cite{BKL16,GKL18} for the Navier--Stokes--Cahn--Hilliard system describing the dynamics of binary fluid mixtures.

From the energetic point of view, the system \eqref{NS}--\eqref{pot2} exhibits competitions between the macroscopic kinetic energy for the fluid and the microscopic elastic bending energy for the functionalized membrane or mixtures with an amphiphilic structure, that is, the coupling involves different scales (macroscopic hydrodynamics with mesoscopic interfacial morphology).
To begin, we assume the energy law
$$
\frac{\mathrm{d}}{\mathrm{d}t} E^{\text{tot}}=-\mathcal{D}^{\text{tot}}.
$$
The total energy combines the kinetic energy with the elastic energy such that
$$
E^{\text{tot}}(\uu,\phi)=\frac12\int_\Omega |\uu |^2\,\mathrm{d}x +E(\phi),
$$
where $E(\phi)$ is given by \eqref{defiE}. On the other hand, the total energy dissipation are assumed to consist of two contributions, from the fluid viscosity and dissipation on the interface (due to relative drag in different scales), that is (cf. \cite{BKL16})
$$
\mathcal{D}^{\text{tot}}= \int_\Omega 2\nu(\phi)|D\uu|^2 \,\mathrm{d}x  +\int_\Omega \frac{\phi^2}{m(\phi)} |\mathbf{v}-\mathbf{u}|^2\,\mathrm{d}x.
$$
Here, $\uu$ denotes the macroscopic fluid velocity of the fluid and $\mathbf{v}$ denotes the ``effective velocity'' for the phase-field function, which includes the microscopic effect due to diffusion. They are assumed to be independent due to scale separation. Besides, we do not consider boundary effects here and simply assume that
\begin{align}
\uu=\mathbf{0},\quad \partial_\mathbf{n}\phi=\partial_\mathbf{n}\Delta \phi=\mathbf{v}\cdot\mathbf{n}=0,\quad \text{on}\ \partial\Omega.\label{DBB}
\end{align}

\textit{(1) Derivation at the macroscopic level}. The motion is related to the flow map $x(X, t)$ determined by macroscopic fluid velocity $\uu$ such that
\be
\frac{\mathrm{d}x}{\mathrm{d}t}=\uu(x(X, t), t), \;\; x(X, 0)=X.\non
\ee
The deformation tensor $\mathsf{F}$ associated with the flow map
is given by $\mathsf{F}_{ij}=\frac{\partial x_i}{\partial X_j}$. Here we have
${\rm det}\mathsf{F}=1$, which is equivalent to the incompressibility condition $\nabla \cdot \uu=0$.
At the macroscopic level, the phase function $\phi$ is purely transported by the velocity $\uu$, namely,
$$
\partial_t \phi+\uu\cdot\nabla\phi=0.
$$
The action functional $\mathcal{A}$ is given via the Legendre transform of the total energy
 \bea
\mathcal{A}(x)&=&\int_0^T\int_{\Omega_0^X}\Big[\frac12|x_t(X,t)|^2-E(\phi(x(X,t),t))\Big]{\rm det}\mathsf{F}\,\mathrm{d}X\mathrm{d}t,\notag
\eea
 with $\Omega_0^X$ being the original domain. In order to perform variation under the constraint $\nabla \cdot \uu=0$, we can take a one-parameter family of volume preserving flow maps (see \cite{DLRW09,HKL,WX13})
\[
x^s(X, t) \ \ \mbox{with} \ \ x^0(X,t)=x(X,t), \ \ \frac{\mathrm{d}}{\mathrm{d}s}x^{s}(X,t)\Big|_{s=0}=\mathbf{w}(X,t),
\]
where $\mathbf{w}$ is an arbitrary divergence free vector field. The velocity $\tilde{\mathbf{w}}$ associated with $\mathbf{w}$ in the Eulerian coordinates is defined by $\tilde{\mathbf{w}}(x(X,t),t)=\mathbf{w}(X,t)$.
The volume-preserving condition is equivalent to ${\rm det}\mathsf{F}^s={\rm det}(\nabla_X x^s) =1$.
In addition, given a function $\phi_0 : \Omega_0^X \rightarrow \mathbb{R}$, we define the perturbation of $\phi$ by
$\phi^s : \Omega^s \times [0, T]\rightarrow \mathbb{R}$, which takes the constant value $\phi_0(X)$ along the particle trajectory
starting from $X$ under the motion $x^s$ such that
 \be
\phi^s(x^s(X, t), t)=\frac{\phi_0(X)}{{\rm det}\mathsf{F}^s}=\phi_0(X). \non
 \ee
Moreover, by definition, we have $\phi^0=\phi$.
 Using the chain rule, we find that (see \cite{DLRW09})
 $$
 \partial_{s}{\phi^s}(x^s)+\left(\frac{\mathrm{d}}{\mathrm{d}s}{x}^{s}\right)\cdot \nabla_{x^s}\phi^s(x^s)=\frac{\mathrm{d}}{\mathrm{d}s}\phi_0(X)=0.
 $$
 Taking $s=0$, we get
 \be
\partial_{s}{\phi^s}(x^s) \Big|_{s=0}+\tilde{\mathbf{w}}\cdot\nabla_x\phi=0. \notag
 \ee

The least action principle states that the trajectories of particles from the position
$x(X,0)$ at time $t=0$ to $x(X,T)$ at a given time $T$ in a Hamiltonian system are those
that minimize the action functional $\mathcal{A}$. Taking variation of the action functional $\mathcal{A}$ with respect to the flow map $x$ leads to the conservative force (cf. \cite[Section 2.2.1]{SV}) such that $\delta_x\mathcal{A}:=(F_{\text{inertial}}+F_{\text{conv}})\cdot \delta x$. Here, for every smooth, one-parameter family of motions $x^s$ with $x^s(\cdot, 0)= x(\cdot,0)$ and $x^s(\cdot, T)=x(\cdot,T)$, we have (pushing forward some part to the Eulerian coordinates just for simplicity of presentation)
\bea
\delta_x\mathcal{A}=\frac{\mathrm{d}}{\mathrm{d}s}\mathcal{A}(x^s)\Big|_{s=0}
&=&\int_0^T \left(\int_{\Omega_0^X}x_t\cdot
\mathbf{w}_t\, \mathrm{d}X
-\int_\Omega \frac{\delta{E}}{\delta\phi}\partial_s\big(\phi^s(x^s)\big)\Big|_{s=0}\, \mathrm{d}x\right) \mathrm{d}t  \non\\
&=&-\int_0^T\int_{\Omega}\Big(\partial_t \uu+\uu\cdot\nabla{\uu}
-\frac{\delta{E}}{\delta\phi}\nabla_x\phi\Big)\cdot\tilde{\mathbf{w}}\,\mathrm{d}x\mathrm{d}t. \non
\eea
Since $\tilde{\mathbf{w}}$ is arbitrary, we can formally write down the
inertial and conservative forces (in the strong form)
\be
F_{\text{inertial}}+F_{\text{conv}}=-(\partial_t \uu+\uu\cdot\nabla{\uu}+\nabla P_1)+\frac{\delta{E}}{\delta\phi}\nabla\phi,
\notag
\ee
where $P_1$ is understood as a Lagrangian multiplier for the incompressibility condition.

Next, we derive dissipative force by applying Onsager's maximum dissipation principle \cite{O31,O31-2}.
As in \cite{BKL16}, because of scale separation assumption, we only take variation of the dissipation functional due the fluid viscosity
with respect to $\uu$. Thus, for the Rayleigh dissipation functional $\mathcal{R}(\uu):= \int_\Omega \nu(\phi)|D\uu|^2 \,\mathrm{d}x$ (which is half of the original dissipation),  we can derive the generalized dissipative force by noticing that  $\delta_\uu  \mathcal{R}:=-F_{\text{diss}}\cdot\delta\uu$ (cf. \cite[Section 2.2.1]{SV}). Let $\mathbf{w}$ be an arbitrary smooth vector function with $\nabla\cdot{\mathbf{w}}=0$.
We have
\bea \delta_\uu\mathcal{R}(\uu)=\frac{\mathrm{d}\mathcal{R}(\uu+s{\mathbf{w}})}{\mathrm{d}s}\Big|_{s=0}
  &=& \int_{\Omega}(2\nu(\phi)D{\uu}) : D{\mathbf{w}}\,\mathrm{d}x\non\\
  &=&-\int_{\Omega}\nabla \cdot (2\nu(\phi)D{\uu})\cdot \mathbf{w}\,\mathrm{d}x, \non
\eea
which yields the dissipative force (formally written in the strong form)
\be
F_{\text{diss}}= \nabla \cdot (2\nu(\phi)D{\uu}) +\nabla{P}_2,
\notag
\ee
where $P_2$ is also a Lagrangian multiplier for the incompressibility condition.

From the classical Newton's force balance law $F_{\text{inertial}}+F_{\text{conv}}+F_{\text{diss}}=\mathbf{0}$, we arrive at the momentum equation
\begin{align}
&\partial_t \uu+\uu\cdot\nabla{\uu}-\nabla \cdot (2\nu(\phi)D{\uu}) +\nabla P= \frac{\delta{E}}{\delta\phi}\nabla\phi, \label{D1}\\
&\nabla \cdot\uu=0,\label{D2}
\end{align}
where $P=P_1-P_2$ is a Lagrangian multiplier for the constraint $\nabla\cdot \mathbf{u}=0$. Besides, we denote the chemical potential by
\begin{align}
\mu=\frac{\delta{E}}{\delta\phi} = -\Delta \omega + f'(\phi)\omega + \eta\omega, \quad \omega = -\Delta \phi + f(\phi).
\label{D3}
\end{align}

\textit{(2) Derivation at the microscopic level}. The motion is related to the flow map $x(X, t)$ now determined by the effective velocity $\mathbf{v}$:
\be
   \frac{\mathrm{d}x}{\mathrm{d}t}=\mathbf{v}(x(X, t), t), \;\; x(X, 0)=X.\non
\ee
Note that here we do not impose the constraint $\nabla \cdot\mathbf{v}=0$ for $\mathbf{v}$.
Concerning the kinematics, the evolution of the phase function is assumed to satisfy a continuity equation (see \cite{BKL16}):
\begin{align}
\partial_t\phi+\nabla \cdot(\phi\mathbf{v})=0,\notag
\end{align}
which yields that $\phi$ is a conserved order parameter as required. By computations similar to \cite[Appendix]{LW19}, neglecting the kinetic energy due to the assumption of scale separation, we can take the variation of the reduced action functional
$$
\widehat{\mathcal{A}}(x)=-\int_0^T\int_{\Omega_0^X}E(\phi(x(X,t),t)){\rm det}\mathsf{F}\,\mathrm{d}X\mathrm{d}t,
$$
and obtain
$$
\delta_x\widehat{\mathcal{A}}=\frac{\mathrm{d}}{\mathrm{d}s}\widehat{\mathcal{A}}(x+sy)\Big|_{s=0}=-\int_0^T\int_\Omega \phi\left(\nabla \frac{\delta{E}}{\delta\phi}\right) \cdot \tilde{y}\,\mathrm{d}x\mathrm{d}t,
$$
where $\tilde{y}$ is an arbitrary smooth vector $y(X,t)=\tilde{y}(x(X,t),t)$ satisfying $\tilde{y}\cdot \mathbf{n}=0$. In this way, we derive the generalized conservative force (formally written in the strong form)
$$
\widehat{F}_{\text{conv}}= - \phi\nabla \frac{\delta{E}}{\delta\phi}=-\phi\nabla \mu.
$$
Next, by applying Onsager's maximum dissipation principle to the dissipation function $\widehat{\mathcal{R}}(\mathbf{v})= \frac12\int_\Omega \frac{\phi^2}{m(\phi)} |\mathbf{v}-\mathbf{u}|^2\,\mathrm{d}x$, we obtain that
\bea \delta_\mathbf{v}\widehat{\mathcal{R}}(\mathbf{v})=\frac{\mathrm{d}\widehat{\mathcal{R}}(\mathbf{v}+s\mathbf{w})}{\mathrm{d}s}\Big|_{s=0}
  &=& \int_{\Omega}\frac{\phi^2}{m(\phi)} (\mathbf{v}-\mathbf{u}) \cdot \mathbf{w}\,\mathrm{d}x,
 \non
  \eea
which yields the dissipative force (formally written in the strong form)
\be
\widehat{F}_{\text{diss}}= -\frac{\phi^2}{m(\phi)} (\mathbf{v}-\mathbf{u}),
\notag
\ee
Using the force balance relation  $\widehat{F}_{\text{inertial}}+\widehat{F}_{\text{conv}}+\widehat{F}_{\text{diss}}=\mathbf{0}$ again (with $\widehat{F}_{\text{inertial}}=\mathbf{0}$ at the microscopic level), we can deduce that
$$
\phi\mathbf{v}=\phi\uu-m(\phi)\nabla\mu.
$$
Inserting the above expression into the continuity equation for $\phi$ and using $\nabla \cdot\uu=0$, we can deduce the Cahn--Hilliard type equation with convection
\begin{align}
\partial_t \phi+\uu\cdot\nabla \phi=\nabla \cdot(m(\phi)\nabla \mu).
\label{D4}
\end{align}

\textit{Summary}. In view of \eqref{D1}--\eqref{D4}, we finally arrive at the system \eqref{NS}--\eqref{pot2}. Besides, from \eqref{DBB} and the expression of $\mathbf{v}$, we find that $\partial_\mathbf{n}\mu=0$ on $\partial\Omega$, provided that $m(\phi)$ is not degenerate. Hence, the boundary condition \eqref{boundary} is also recovered.

\bigskip
\noindent \textbf{Acknowledgments:}
The first author was partially supported by NSF of China 12071084 and the Shanghai Center for Mathematical Sciences at Fudan University.


\end{document}